\documentclass{article}

\usepackage[english]{babel}
\usepackage[letterpaper,top=2cm,bottom=2cm,left=3cm,right=3cm,marginparwidth=1.75cm]{geometry}
\usepackage[utf8]{inputenc}

\usepackage{afterpage}
\usepackage{amsfonts,amsmath,amssymb,amsthm,mathtools}
\usepackage{mathrsfs}

\usepackage{algorithm}
\usepackage{algorithmic}
\usepackage{appendix}
\usepackage{booktabs}
\usepackage{cite}
\usepackage{enumerate}
\usepackage{graphicx}
\usepackage{lscape}
\usepackage{multirow}
\usepackage{placeins}
\usepackage[figuresright]{rotating}

\usepackage[dvipsnames]{xcolor}
\usepackage[
  colorlinks=true,
  allcolors=NavyBlue,
  urlcolor=Black
]{hyperref}
\hypersetup{
  pdftitle={Convergence Analysis of the Restarted Moving-Anchored Extra-Gradient Method in the Absence of Local Lipschitz Continuity},
  pdfauthor={Defeng Sun, Liping Zhang, and Wei Zhao}
}
\usepackage{cleveref}

\theoremstyle{plain}
\newtheorem{theorem}{Theorem}
\newtheorem{lemma}{Lemma}
\newtheorem{proposition}{Proposition}
\newtheorem{assumption}{Assumption}

\theoremstyle{definition}
\newtheorem{definition}{Definition}
\newtheorem{example}{Example}

\theoremstyle{remark}
\newtheorem{remark}{Remark}

\title{Convergence Analysis of the Restarted Moving-Anchored Extra-Gradient Method in the Absence of Local Lipschitz Continuity%
\thanks{Submitted to the editors DATE. Funding: The research of Defeng Sun was supported in part by the Hong Kong RGC Senior Research Fellow Scheme No. SRFS22235S02 and the Research Center for Intelligent Operations Research and the research of Liping Zhang was supported in part by the National Natural Science Foundation of China under Grant No. 12571323.}}
\author{
Defeng Sun\thanks{Department of Applied Mathematics, The Hong Kong Polytechnic University, Hung Hom, Hong Kong. E-mail: \href{mailto:defeng.sun@polyu.edu.hk}{\nolinkurl{defeng.sun@polyu.edu.hk}}.}
\and
Liping Zhang\thanks{Department of Mathematical Sciences, Tsinghua University, Beijing, 100084, China. E-mail: \href{mailto:lipingzhang@tsinghua.edu.cn}{\nolinkurl{lipingzhang@tsinghua.edu.cn}}.}
\and
Wei Zhao\thanks{Department of Mathematical Sciences, Tsinghua University, Beijing, 100084, China; Department of Applied Mathematics, The Hong Kong Polytechnic University, Hung Hom, Hong Kong. E-mails: \href{mailto:zhaow22@mails.tsinghua.edu.cn}{\nolinkurl{zhaow22@mails.tsinghua.edu.cn}}, \href{mailto:wei22.zhao@polyu.edu.hk}{\nolinkurl{wei22.zhao@polyu.edu.hk}}.}
}

\begin{document}
\maketitle

\begin{abstract}
In this paper, we introduce the moving-anchored extra-gradient (MAEG) method for solving monotone inclusion problems involving the sum of a continuous monotone operator and a maximal monotone operator. Notably, the distance from the anchor point to the solution set is designed to be monotonically non-increasing. Under Lipschitz continuity of the forward operator, MAEG attains an $\mathcal{O}(1/k)$ non-asymptotic iteration complexity, and when a positive anchor-update parameter is used, it further achieves an $o(1/k)$ asymptotic rate. Furthermore, leveraging the specific behavior of the anchor point, we propose a tailored restart strategy. We demonstrate that this strategy ensures convergence even in the absence of local Lipschitz continuity, while preserving the original iteration complexity guarantees whenever the Lipschitz condition holds.
\end{abstract}



\section{Introduction}
In recent years, first-order methods (FOMs) have achieved significant success in solving large-scale optimization problems, particularly in linear programming~\cite{applegate2021practical,applegate2023faster,chen2025hpr} and convex quadratic programming~\cite{chen2025hprqp}. Their key advantage lies in the low cost and high parallelizability of each iteration, making them well-suited for modern high-performance architectures like GPUs. 

In this paper, we focus on developing efficient FOMs for solving the monotone inclusion problem:
\begingroup
\renewcommand{\theHequation}{problem-mi}
\begin{equation}\tag{MI}\label{eq:mi}
    \text{find } x \in \mathbb{R}^n \text{ such that } 0 \in T(x)= (F+B)(x),
\end{equation}
\endgroup
where $ F: \operatorname{dom} F\subseteq \mathbb{R}^n \to \mathbb{R}^n $ is a monotone, continuous operator, and $ B : \mathbb{R}^n \rightrightarrows \mathbb{R}^n $ is a set-valued maximal monotone operator with $\operatorname{cl}(\operatorname{dom} B)\subseteq \operatorname{dom} F$. We assume throughout that the solution set of~\eqref{eq:mi}, denoted by $\operatorname{Sol}(F,B)$, is nonempty. 

It is worth emphasizing that our problem setting is quite general. First, we do not assume that $F$ is inverse strongly monotone
(equivalently, co-coercive), a condition commonly used in the analysis
of classical forward-backward splitting methods~\cite{lions1979splitting}. A commonly adopted but weaker assumption is that $ F(\cdot) $ is Lipschitz continuous. However, in this work, we do not require $ F(\cdot) $ to be Lipschitz continuous, either globally or even locally.

The problem~\eqref{eq:mi} can be seen as a generalization of the variational inequality (VI) problem. Let $ C \subseteq \mathbb{R}^n$ be a nonempty closed convex set and denote the normal cone mapping of $C$ as $N_C$. 
When $ B = N_C$, the problem~\eqref{eq:mi} reduces to the classical monotone VI problem:
\begingroup
\renewcommand{\theHequation}{problem-vi}
\begin{equation}\tag{VI}\label{eq:vi}
    \text{find } x_\star \in C \text{ such that } \langle F(x_\star), x - x_\star \rangle \geq 0, \text{ for all } x \in C.
\end{equation}
\endgroup
We briefly review significant algorithmic developments for solving problem~\eqref{eq:vi} in recent decades. Notably, these algorithms can be extended to address problem~\eqref{eq:mi} without substantial technical barriers. In 1976, further assuming that \(F(\cdot)\) is Lipschitz continuous with Lipschitz constant \(L\), Korpelevich~\cite{korpelevich1976extragradient} used the idea of extrapolation and proposed the well-known extra-gradient (EG) method:
\begin{equation}\label{eq:eg}
    \bar{x}_k=\Pi_C\left(x_k-\beta F\left(x_k\right)\right),\quad x_{k+1}=\Pi_C\left(x_k-\beta F\left(\bar{x}_k\right)\right),
\end{equation}
where $\beta\in(0,\frac{1}{L})$. Following Korpelevich's seminal work, many EG variants were proposed, such as the optimistic gradient descent-ascent (OGDA)~\cite{popov1980modification}, the projected reflected gradient (PRG) method~\cite{malitsky2015projected}, and the golden ratio (GR) algorithm~\cite{malitsky2020golden}. When \(F(\cdot)\) is not globally Lipschitz, Khobotov~\cite{khobotov1987modification} adapted EG by incorporating a backtracking line search, obtaining convergence under the assumption of local Lipschitz continuity\footnote{In Khobotov's paper, \(F(\cdot)\) is only assumed to be monotone and continuous. We provide a counterexample in the appendix to show that the proof does not hold without the assumption of local Lipschitz continuity.}. In~\cite{sun1993projected}, Sun proved convergence for EG with line search under only continuity and monotonicity assumptions, without additional conditions. Later, more adaptive variants of EG were proposed to solve~\eqref{eq:vi}, such as the projection and contraction (PC) method~\cite{sun1994projection,sun1995new,sun1996class}.
The underlying principles of the PC method were later extended to~\eqref{eq:mi} via the modified forward-backward splitting (MFBS) method~\cite{tseng2000modified}.

We now introduce some iteration complexity results of EG-type methods for solving problem~\eqref{eq:vi} in terms of the standard natural residual, defined as
\begin{equation}\label{eq:natural-residual}
     \mathcal{R}_{\mathrm{nat}}(x) = x - \Pi_C\left(x - F(x)\right).
\end{equation}
For the complexity results reviewed below, we assume \(F(\cdot)\) is Lipschitz continuous with constant \(L\).
It is well-known that the best iteration complexity of the EG method is \(\mathcal{O}(1/\sqrt{k})\)~\cite{korpelevich1976extragradient}. In~\cite{nemirovski2004prox}, Nemirovski analyzed the averaged iterates of the EG method~\eqref{eq:eg}, denoted by \(\bar{x}^k = \frac{1}{k} \sum_{i=1}^k \bar{x}_i\). Assuming that the feasible set \(C\) is bounded, Nemirovski established an \(\mathcal{O}(1/k)\) convergence rate with respect to the gap function, namely,
\begin{equation*}
    \sup_{y \in C}\langle F(y), \bar{x}^k - y \rangle = \mathcal{O}\left(\frac{1}{k}\right).
\end{equation*}
To remove the boundedness condition on \(C\), Monteiro and Svaiter~\cite{monteiro2010complexity, monteiro2011complexity} demonstrated that Korpelevich’s EG method can be viewed as a special case of the hybrid proximal extra-gradient (HPE) method proposed by Solodov and Svaiter~\cite{solodov1999hybrid}. They established the following bound:
\begin{equation}\label{eq:weak-residual}
    \sup _{y \in C}\langle F(y) - \bar{r}_k, \bar{x}^k - y \rangle =  \mathcal{O}\left(\frac{1}{k}\right),\quad \text{ with }\|\bar{r}_k\| = \mathcal{O}\left(\frac{1}{k}\right),
\end{equation}
where \(\bar{r}_k\) is a quantity associated with the iterates. However, this bound~\eqref{eq:weak-residual} only implies an \(\mathcal{O}(1/\sqrt{k})\) convergence rate for the natural residual~\cite[Theorem A.4 and Corollary A.2]{monteiro2010complexity}. To the best of our knowledge, whether the ergodic sequence \(\{\bar{x}^k\}\) can achieve an \(\mathcal{O}(1/k)\) convergence rate in terms of the natural residual is still unknown.

On the other hand, significant advancements have been made in analyzing the complexity of fixed-point problems. In~\cite{lieder2021convergence}, Lieder identified the optimal parameters for Halpern iteration~\cite{halpern1967fixed, wittmann1992approximation, sabach2017first}, achieving an $\mathcal{O}(1/k)$ convergence rate for the fixed-point residual for general nonexpansive operators. The Halpern iteration with optimal parameters has been used in the design of accelerated optimization algorithms~\cite{kim2021accelerated, zhang2022efficient, yang2025accelerated, sun2025accelerating}. However, the Halpern iteration cannot be directly applied to accelerate fixed-point problems of the form $x = \Pi_C ( x - \beta F( x ))$, as the operator \(\Pi_C(x - \beta F(x))\) does not meet the nonexpansiveness condition unless \(F(\cdot)\) is co-coercive. To address this limitation, Yoon and Ryu~\cite{yoon2021accelerated} integrated the EG method with an anchoring technique from Halpern iteration, introducing the extra-anchor-gradient (EAG) method for solving~\eqref{eq:vi} in the unconstrained case \(C = \mathbb{R}^n\). The EAG method is iteratively defined by
\begin{equation}\label{eq:eag}
    \bar{x}_k = \frac{1}{k+2}x_0 + \frac{k+1}{k+2}x_k - \alpha_k F(x_k),\quad 
    x_{k+1} = \frac{1}{k+2}x_0 + \frac{k+1}{k+2}x_k - \alpha_k F(\bar{x}_k).
\end{equation}
Although the stepsize sequence \(\{\alpha_k\}\) should be chosen restrictively in their analysis, EAG is the first method to achieve the \(\mathcal{O}(1/k)\) convergence rate for the natural residual, matching the theoretical optimal order for this class of problems. Building on this, Lee and Kim~\cite{lee2021fast} proposed a variant of EAG known as the fast extra-gradient (FEG) method, which is defined as:
\begin{equation}\label{eq:feg}
    \bar{x}_k = \frac{1}{k+1}x_0 + \frac{k}{k+1}(x_k - \alpha F(x_k)),\quad 
    x_{k+1} = \frac{1}{k+1}x_0 + \frac{k}{k+1}x_k - \alpha F(\bar{x}_k),
\end{equation}
where \(\alpha \in (0, 1/L)\). While maintaining the \(\mathcal{O}(1/k)\) convergence rate, the stepsize in FEG can be selected more flexibly.
Sequence convergence properties of both EAG and FEG are explored in~\cite{yoon2025accelerated} by analyzing the relationship between the trajectories they generated and the trajectory generated by Halpern accelerated PPA~\cite{kim2021accelerated}. Furthermore, extensions of these complexity and convergence results to~\eqref{eq:mi} have been investigated in~\cite{tran2023sublinear, cai2024accelerated}.

More recently, motivated by the potential for further acceleration in practice, accelerated algorithms with moving anchor points have been proposed~\cite{alcala2023moving, yuan2025symplectic,yuan2024symplectic}. In~\cite{yuan2024symplectic}, Yuan and Zhang analyzed the symplectic extra-gradient (SEG) method based on the symplectic discretization approach of certain high-resolution ODEs, which is the first extra-gradient-type method that achieves both a non-asymptotic \(\mathcal{O}(1/k)\) convergence rate and an asymptotic \(o(1/k)\) convergence rate. The iteration of SEG method is given by
\begin{equation}\label{eq:seg}
    \bar{x}_k = \frac{r}{k+r}u_k + \frac{k}{k+r}(x_k - \alpha F(x_k)),\quad 
    x_{k+1} = \frac{r}{k+r}u_k + \frac{k}{k+r}x_k - \alpha F(\bar{x}_k),
\end{equation}
where the anchor point $\{u_k\}$ is iteratively updated by $u_0 = x_0$ and $u_{k+1} = u_k - \frac{C}{r}\alpha F(x_{k+1})$ with $\alpha\in(0,1/L)$, $r\in(1,+\infty)$ and $C\in(0,r-1)$. In addition to the above complexity results, the convergence of both \(\{x_k\}\) and \(\{u_k\}\) was also established in ~\cite{yuan2024symplectic} under a fixed stepsize.

While the accelerated extra-gradient-type algorithms enjoy improved iteration complexity bounds, it is important to highlight that all the accelerated methods discussed above require the global Lipschitz continuity condition on the operator $F(\cdot)$. This is a strong assumption, which limits the class of problems that can be addressed, in contrast to the extra-gradient-type method without acceleration. The aim of this paper is to eliminate the Lipschitz continuity condition in accelerated algorithms. To achieve this, we introduce a new moving-anchor-based accelerated algorithm, called the moving-anchored extra-gradient (MAEG) method. The key difference between the MAEG method and other moving-anchor-based methods lies in the behavior of the anchor point: in the MAEG method, the distance between the anchor point and the solution set is non-increasing. In contrast, other methods lack this property. Leveraging this non-increasing property, we further propose a restart strategy. The MAEG method equipped with this strategy is referred to as MAEG-R. Both the MAEG and MAEG-R methods achieve an iteration complexity of $\mathcal{O}(1/k)$ when the operator $F(\cdot)$ is Lipschitz continuous. Notably, the MAEG-R method remains convergent even when the operator $F(\cdot)$ is neither globally nor locally Lipschitz continuous. The main contributions of this paper can be highlighted as follows:
\begin{enumerate}[(i)]
    \item We propose the MAEG method and prove the iteration complexity and convergence results under mild conditions. Specifically, when the operator $F(\cdot)$ is Lipschitz continuous, the MAEG method achieves the non-asymptotic $\mathcal{O}(1/k)$ convergence rate in terms of the composite natural residual. Moreover, with a positive anchor-update parameter, it further enjoys an asymptotic $o(1/k)$ rate.

    \item We propose a new restart strategy tailored to the MAEG method. With the new restart strategy, the MAEG remains convergent even for non-Lipschitz monotone inclusion problems while preserving the $\mathcal{O}(1/k)$ and $o(1/k)$ convergence rates in the Lipschitz case.        
    
\end{enumerate}

The remainder of this paper is organized as follows. In Section 2, we introduce some necessary preliminaries. In Section 3 and Section 4, we propose the MAEG method and the MAEG-R method respectively. Section 5 presents numerical results demonstrating the effectiveness of our algorithms. Finally, we conclude the paper in Section 6.

\noindent \textbf{Notation.} Let $\Pi_C(\cdot)$ denote the metric projection operator onto a nonempty, closed, and convex set $C$. Additionally, we use $\mathbb{B}(x, r)$ to represent the closed ball centered at $x$ with radius $r$.

\section{Preliminaries and Technical Lemmas}
In this section, we introduce the basic concepts related to problem~\eqref{eq:mi} and establish some technical lemmas that form the foundation for the subsequent analysis.

A set-valued operator $T\colon\mathbb{R}^n \rightrightarrows \mathbb{R}^n$ is called monotone if, for any $x,x'\in\mathbb{R}^n$, $v\in T(x)$, and $v'\in T(x')$, one has $\langle v-v',x-x'\rangle \ge 0$.
It is said to be maximal monotone if, in addition, the graph of the operator $ T(\cdot) $, $\operatorname{gph}(T)=\{(x, v) \in \mathbb{R}^n \times \mathbb{R}^n \mid v \in T(x)\}$,
is not properly contained in the graph of any other monotone operator. The resolvent of any monotone operator $T$ is the set-valued operator defined as $J_T = (\operatorname{Id} + T)^{-1}$, where $\operatorname{Id}$ denotes the identity operator on $\mathbb{R}^n$. If $T$ is maximal monotone, then $J_T$ is single-valued, defined on the entire space $\mathbb{R}^n$, and firmly nonexpansive, i.e. $\|J_T(x) - J_T(x')\|^2 \leq \langle J_T(x) - J_T(x'), x - x' \rangle$, for any $x, x' \in \mathbb{R}^n$.

In this work, we focus on the monotone operator $T=F+B$. The assumptions on $F$ and $B$ in the problem setting ensure, by
\cite[Proposition A.1]{monteiro2011complexity}, that $T$ is maximal monotone. Consequently, its zero set $\operatorname{Sol}(F,B)=T^{-1}(0)$ is closed and convex. To measure the quality of approximate solutions to~\eqref{eq:mi}, we introduce the notion of the composite natural residual mapping. 

\begin{definition}
Let $F \colon \operatorname{dom} F \subseteq \mathbb{R}^n \to \mathbb{R}^n$ be a monotone and continuous operator, and let $B \colon \mathbb{R}^n \rightrightarrows \mathbb{R}^n$ be a maximal monotone operator. The composite natural residual mapping associated with~\eqref{eq:mi} is defined by
\begin{equation*}
    \mathcal{R}_{\mathrm{nat}}^c(x, \lambda) \coloneqq \frac{1}{\lambda}\left( x - J_{\lambda B} \big( x - \lambda F(x) \big) \right), \quad \forall x \in \operatorname{dom} F, \, \lambda > 0.
\end{equation*}
When $\lambda = 1$, we abbreviate $\mathcal{R}_{\mathrm{nat}}^c(x, 1)$ as $\mathcal{R}_{\mathrm{nat}}^c(x)$.
\end{definition}

Note that $x_\star \in \operatorname{Sol}(F,B)$ if and only if $\mathcal{R}_{\mathrm{nat}}^c(x_\star, \lambda) = 0$ for any $\lambda > 0$. Next, we present a lemma concerning the continuity of the resolvents $J_{\lambda B}(x)$ as a function of both $\lambda$ and $x$.

\begin{lemma}\label{lem:resolvent-continuity}
Let $ B \colon \mathbb{R}^n \rightrightarrows \mathbb{R}^n $ be a maximal monotone operator. Then the following statements hold:
\begin{enumerate}[(i)]
    \item $\operatorname{cl}(\operatorname{dom} B)$ is a nonempty, closed and convex set.
    \item For any $x \in \mathbb{R}^n$, $J_{\lambda B}(x) \to \Pi_{\operatorname{cl}(\operatorname{dom} B)}(x)$ as $\lambda \to 0^+$.
    \item The mapping $\varphi(\cdot,\cdot)\colon\mathbb{R}^n\times \mathbb{R}_+\to \mathbb{R}^n $ defined by
    \begin{equation}\label{eq:varphi}
    \varphi(x,\lambda) = 
    \begin{cases}
    J_{\lambda B}(x), & \text{if } \lambda > 0 \\
    \Pi_{\operatorname{cl}(\operatorname{dom} B)}(x), & \text{if } \lambda = 0
    \end{cases}
    \end{equation}
    is jointly continuous on $\mathbb{R}^n \times \mathbb{R}_+$.
\end{enumerate}
\end{lemma}
\begin{proof}
Statements \textnormal{(i)} and \textnormal{(ii)} follow directly from~\cite[Corollary~21.14]{bauschke2017convex} and~\cite[Theorem~23.48]{bauschke2017convex}, respectively.  
Statement \textnormal{(iii)} is a consequence of~\cite[Proposition~3.4]{atenas2025relocated}.
\end{proof}

The following lemma is a generalization of~\cite[Lemma 1]{gafni1984two} and~\cite[Lemma 2.2]{calamai1987projected}.
\begin{lemma}\label{lem:nonincreasing-resolvent}
Let $B:\mathbb{R}^n\rightrightarrows\mathbb{R}^n$ be a maximal monotone operator. Then, for any fixed $x,d\in\mathbb{R}^n$, the function $\phi:(0,\infty)\to\mathbb{R}$ defined by
$\phi(\lambda):=\|J_{\lambda B}(x+\lambda d)-x\|/\lambda$
is non-increasing.
\end{lemma}
\begin{proof}
Let $\lambda_1 > \lambda_2 > 0$ be given. Denote $z_1 = J_{\lambda_1 B}(x + \lambda_1 d)$ and $z_2 = J_{\lambda_2 B}(x + \lambda_2 d)$. If $z_1 = z_2$, then clearly $\phi(\lambda_1) \leq \phi(\lambda_2)$. Thus, we assume $z_1 \neq z_2$. We will use the following implication to prove the result:
\begin{equation}\label{eq:ratio-bound}
   \text{for any } a,b\in \mathbb{R}^n , \langle b, a - b \rangle > 0 \implies \frac{\| a \|}{\| b \|} \leq \frac{\langle a , a - b \rangle}{\langle b , a - b \rangle}.
\end{equation}
Define $a = z_1 - x$, $b = z_2 - x$, $v_1 = d - \frac{a}{\lambda_1} \in B(z_1)$ and $v_2 = d - \frac{b}{\lambda_2} \in B(z_2)$. Using the monotonicity of $B(\cdot)$, we obtain $ \left\langle \frac{b}{\lambda_2} - \frac{a}{\lambda_1}, a - b \right\rangle = \left\langle v_1 - v_2, z_1 - z_2\right\rangle\ge 0 $, which implies
$\langle b, a - b \rangle \ge \frac{\lambda_2}{\lambda_1} \langle a, a - b \rangle$.

Furthermore, let $z_3 = z_1 + \lambda_2 v_1$. Then we have $z_1 = J_{\lambda_2 B}(z_3)$. Using the firm nonexpansiveness of the resolvent $J_{\lambda_2 B}(\cdot)$, we have $\| z_1 - z_2 \|^2 \leq \langle z_1 - z_2, z_3 - (x + \lambda_2 d) \rangle = \left( 1 - \frac{\lambda_2}{\lambda_1} \right) \langle a - b, a \rangle$. Since $z_1 \neq z_2$, it follows that $\langle b, a - b \rangle \ge \frac{\lambda_2}{\lambda_1} \langle a, a - b \rangle \ge \frac{\lambda_2}{\lambda_1 - \lambda_2} \| z_1 - z_2 \|^2 > 0$. Finally, applying the observation~\eqref{eq:ratio-bound}, we obtain
\begin{equation*}
    \frac{\| J_{\lambda_1 B}(x + \lambda_1 d) - x \|}{\| J_{\lambda_2 B}(x + \lambda_2 d) - x \|} = \frac{\| a \|}{\| b \|} \leq \frac{\langle a, a - b \rangle}{\langle b, a - b \rangle} \leq \frac{\lambda_1}{\lambda_2},
\end{equation*}
which implies $\phi(\lambda_1) \leq \phi(\lambda_2)$ for $\lambda_1 > \lambda_2$. Hence, $\phi(\cdot)$ is non-increasing.
\end{proof}

We end this section with the following lemma.
\begin{lemma}[Lemma 2.5 in~\cite{xu2002iterative}]\label{lem:xu}
Let $\{s_k\}$ be a sequence of nonnegative real numbers such that
$s_{k+1} \le (1-\alpha_k)s_k+\alpha_k\beta_k+\gamma_k$ for every $k\ge 1$,
where $\alpha_k\in(0,1)$, $\prod_{k=1}^\infty(1-\alpha_k)=0$,
$\limsup_{k\to\infty}\beta_k\le 0$, and $\gamma_k\ge 0$ with
$\sum_{k=1}^\infty\gamma_k<\infty$. Then $s_k\to 0$ as $k\to\infty$.
\end{lemma}

\section{The Moving-Anchored Extra-Gradient Method}

In this section, we present the moving-anchored extra-gradient (MAEG) method for solving~\eqref{eq:mi}, as outlined in Algorithm~\ref{alg:maeg}. We then establish its iteration complexity and convergence results under mild assumptions.

\begin{algorithm}[!htbp]
    \caption{Moving-Anchored Extra-Gradient (MAEG) Method for Solving~\eqref{eq:mi}}
    \label{alg:maeg}
    \begin{algorithmic}[1]
 
    \STATE \textbf{Input:} Initial point $x_0 \in \operatorname{cl}(\operatorname{dom} B)$; parameters $\sigma \in (0,1)$ and $\rho \in \bigl[0, \frac{1}{2}\bigr)$; and a nonempty, closed, convex set $X$ such that $\operatorname{cl}(\operatorname{dom} B) \subseteq X \subseteq \operatorname{dom} F$.

    \STATE \textbf{Initialization:} Set $u_{0} = y_{0} = x_0$, $d_0 = 0$, $\Lambda_0 = 0 $ and $k = 1$.

    \STATE \textbf{Step 1:} Compute $x_k$ and $y_k$:
    \begin{equation}\label{eq:x-update}
        x_{k} = \Pi_X\left(\tau_k u_{k-1} + (1-\tau_k)(y_{k-1} - \lambda_k d_{k-1})\right),
    \end{equation}
    \begin{equation}\label{eq:y-update}
        y_k = J_{\lambda_k B}\left(\tau_k u_{k-1} + (1-\tau_k)y_{k-1} - \lambda_k F(x_{k})\right),
    \end{equation}
    where $\tau_k = \frac{\lambda_k}{(1-2\rho)\Lambda_{k-1} + \lambda_k}$, $\Lambda_k = \sum_{i=1}^{k} \lambda_i$
    and $\lambda_k > 0$ is chosen to satisfy
    \begin{equation}\label{eq:line-search}
        \lambda_{k}\|F(x_{k}) - F(y_{k})\| \leq \sigma\|x_{k} - y_{k}\|.
    \end{equation}

    \STATE \textbf{Step 2:} Update the descent direction and anchor:
    \begin{equation}\label{eq:d-update}
        d_k = \frac{1}{\lambda_k}\left(\tau_k u_{k-1} + (1-\tau_k)y_{k-1}-y_k\right) - F(x_k) + F(y_k),
    \end{equation}
    \begin{equation}\label{eq:u-update}
        u_{k} = u_{k-1} - \rho \lambda_k d_k.
    \end{equation}
    Set $k \gets k+1$ and return to \textbf{Step 1}.

    \STATE \textbf{Output:} The sequence $\{(x_{k}, y_{k}, u_{k})\}$.

    \end{algorithmic}
\end{algorithm}

\begin{remark}
    When $X=\operatorname{dom}F=\mathbb{R}^n$ and $F$ is Lipschitz continuous, Algorithm~\ref{alg:maeg} connects to two existing methods. For $\rho=0$, it has the same anchored extra-gradient structure as the composite Fast Extra-Gradient method~\cite{lee2021fast,tran2023sublinear}, up to the stepsize and weight choices. For $0<\rho<\frac{1}{2}$, it recovers the Symplectic Forward-Backward Splitting method~\cite[Algorithm 3]{yuan2024symplectic} with $r = \frac{1}{1-2\rho}$ and $D = \frac{2\rho}{1-2\rho}$. The projection onto $X$ in~\eqref{eq:x-update} ensures that $F(x_k)$ in~\eqref{eq:y-update} is well-defined; as illustrated in the numerical experiments, a suitable choice of $X$ can significantly improve the efficiency of the line search.
\end{remark}

To ensure that Algorithm~\ref{alg:maeg} is well-defined and to facilitate the subsequent convergence analysis, we impose the following assumptions.

\begin{assumption}\label{asmp:existence}
The solution set of~\eqref{eq:mi}, denoted by $\operatorname{Sol}(F,B)$, is nonempty.    
\end{assumption}

\begin{assumption}\label{asmp:line-search}
For every $k$-th ($k \geq 1$) iteration of Algorithm~\ref{alg:maeg}, there exists a stepsize $\lambda_k > 0$ such that condition~\eqref{eq:line-search} is satisfied. 
\end{assumption}

Since Assumption~\ref{asmp:line-search} is non-trivial, we now examine the conditions under which the assumption holds in the following remark.

\begin{remark}
First, consider the case where $F(\cdot)$ is globally Lipschitz continuous with constant $L$. In this setting, condition~\eqref{eq:line-search} is theoretically guaranteed for any $\lambda_k \in (0, \sigma/L]$. Thus, Assumption~\ref{asmp:line-search} is valid whether one utilizes a fixed stepsize in this range or a standard backtracking line search, which is guaranteed to terminate due to the existence of this feasible interval.
Next, consider the case where $F(\cdot)$ is only locally Lipschitz continuous. For any iterate $y_{k-1}$, let $U_{k-1}$ be a compact neighborhood of $y_{k-1}$ on which $F$ is Lipschitz continuous. Since $x_k \to y_{k-1}$ as $\lambda_k \to 0^+$, we have $x_k \in U_{k-1}$ for sufficiently small $\lambda_k$. Given that $F(x_k)$ is bounded within $U_{k-1}$, it follows from Lemma~\ref{lem:resolvent-continuity}(iii) that $y_k \to \Pi_{\operatorname{cl}(\operatorname{dom} B)}(y_{k-1}) = y_{k-1}$ as $\lambda_k \to 0^+$. Consequently, $y_k$ also falls within $U_{k-1}$ for sufficiently small $\lambda_k$. This guarantees that a backtracking line search will eventually find a valid $\lambda_k > 0$ satisfying~\eqref{eq:line-search}.
Finally, in cases where $F(\cdot)$ fails to satisfy even local Lipschitz continuity, Assumption~\ref{asmp:line-search} cannot be theoretically guaranteed via line search alone. To address this, we propose corresponding restart criteria and a restart strategy designed to maintain convergence, which are detailed in Section~\ref{sec:maeg-r}.
\end{remark}

We employ the norm of the direction $d_k$ as a metric for the optimality of the iteration sequence. 
To validate this choice, we observe from the update rule~\eqref{eq:y-update} and the definition of the resolvent that
\begin{equation*}
    \lambda_k^{-1}\left(\tau_k u_{k-1} + (1-\tau_k)y_{k-1} - \lambda_k F(x_{k}) - y_k \right) \in B(y_k),
\end{equation*}
which implies $d_k \in (F+B)(y_k)$. We now establish the relationship between the standard natural residual $\| \mathcal{R}_{\mathrm{nat}}^c( y_k ) \|$ and $\| d_k \|$ in the following proposition.

\begin{proposition}\label{prop:reduced-gradient-norm}
Suppose Assumption~\ref{asmp:line-search} holds, and let $\{y_k\}$ and $\{d_k\}$ be generated by Algorithm~\ref{alg:maeg}. Then, for all $k \ge 1$, it holds that $ \| \mathcal{R}_{\mathrm{nat}}^c (y_k)\| \leq \| d_k \|$.        
\end{proposition}

\begin{proof}
Since $d_k - F(y_k) \in B(y_k)$, we can write $y_k = J_B(y_k + d_k - F(y_k))$. By the nonexpansiveness of the resolvent, it follows that for all $k \ge 1$:
\begin{equation*}
    \| \mathcal{R}_{\mathrm{nat}}^c(y_k) \| = \| y_k - J_B(y_k - F(y_k)) \| = \| J_B(y_k + d_k - F(y_k)) - J_B(y_k - F(y_k)) \| \leq \| d_k \|,
\end{equation*}
which completes the proof.
\end{proof}

We now proceed to analyze the complexity of the proposed MAEG method. Our analysis relies centrally on the auxiliary function $\mathcal{L}_k$, defined as
\begin{equation}\label{eq:lk}
    \mathcal{L}_k = \langle u_k-y_k, d_k \rangle - \frac{1}{2}(1-2\rho)\Lambda_k\|d_k\|^2 \text{ for all } k \ge 0 .
\end{equation}
The key recursive property of $\mathcal{L}_k$ is established in the following proposition.

\begin{proposition}\label{prop:lk-recursion}
Suppose that Assumption~\ref{asmp:line-search} holds. Let the sequences $\{x_k\}$, $\{y_k\}$, and $\{u_k\}$ be generated by Algorithm~\ref{alg:maeg}, and let $\mathcal{L}_k$ be defined as in~\eqref{eq:lk}. Then, for all $k \ge 1$, it holds that
\begin{equation*}
    \mathcal{L}_k - (1-\tau_k)\mathcal{L}_{k-1} \ge \frac{\lambda_k(1-\sigma^2)}{2\tau_k\sigma^2} \|F(x_k)-F(y_k)\|^2.
\end{equation*}
\end{proposition}

\begin{proof}
Noting $ u_k = u_{k-1} - \rho \lambda_k d_k $ and employing direct computation, we obtain
\begin{equation}\label{eq:lk-difference}
\begin{aligned}
    \mathcal{L}_k - (1-\tau_k)&\mathcal{L}_{k-1} 
    = \langle u_{k-1} - y_k, d_k \rangle - (1-\tau_k) \langle u_{k-1} - y_{k-1}, d_{k-1} \rangle\\
    &\quad-\left(\rho\lambda_k+ \frac{1-2\rho}{2}\Lambda_k\right)\| d_k \|^2+\frac{1-\tau_k}{2}(1-2\rho)\Lambda_{k-1}\|d_{k-1}\|^2.
\end{aligned}
\end{equation}
By definition of $\{d_k\}_{k\ge0}$, the iteration of $\{y_k\}_{k\ge0}$ can be rewritten as
\begin{equation*}
    y_k = \tau_k u_{k-1} + ( 1 - \tau_k ) y_{k-1} - \lambda_k ( d_k + F(x_k) - F(y_k) ), \text{ for all } k \ge 1,
\end{equation*}
which is equivalent to
\[
\left\{
\begin{alignedat}{4}
u_{k-1} - {}& y_k
&= {}& \frac{1-\tau_k}{\tau_k}(y_k-y_{k-1})
+ \frac{\lambda_k}{\tau_k}\bigl(d_k + F(x_k)-F(y_k)\bigr),\\
u_{k-1} - {}& y_{k-1}
&= {}& \frac{1}{\tau_k}(y_k-y_{k-1})
+ \frac{\lambda_k}{\tau_k}\bigl(d_k + F(x_k)-F(y_k)\bigr).
\end{alignedat}
\right.
\]
Substituting the above identities into $\langle u_{k-1} - y_k , d_k \rangle$ and $ \langle  u_{k-1} - y_{k-1} , d_{k-1} \rangle $ respectively, we deduce that
\begin{equation}\label{eq:lk-main-part}
\begin{aligned}
    &\quad\langle u_{k-1} - y_{k}, d_k \rangle - ( 1 - \tau_k ) \langle u_{k-1} - y_{k-1}, d_{k-1} \rangle\\
    &=\frac{1-\tau_k}{\tau_k} \langle y_k - y_{k-1}, d_k - d_{k-1} \rangle + \frac{\lambda_k}{\tau_k}\langle d_k + F(x_k)-F(y_k), d_k -(1-\tau_k)d_{k-1} \rangle.
\end{aligned}
\end{equation}
Since $d_k \in (F+B)(y_k)$ for all $k \ge 1$, it follows from the monotonicity of \((F+B)(\cdot)\) that $\langle y_k-y_{k-1},d_k-d_{k-1}\rangle \ge 0$ for all $k\ge 2$. Combining $\tau_1 = 1$, we have
\begin{equation}\label{eq:monotonicity-gap}
   \frac{1-\tau_k}{\tau_k}\langle y_k-y_{k-1},d_k-d_{k-1}\rangle \ge 0 \text{ for all } k\ge 1. 
\end{equation}
On the other hand, by using the element relation $\langle a,b\rangle = \frac{1}{2}(\| a + b \|^2 - \| a \|^2 - \| b \|^2)$, we obtain
\begin{equation}\label{eq:polarization}
\begin{aligned}
    \langle F(x_k)-F(y_k),\, d_k -(1-\tau_k)d_{k-1} \rangle
    ={}& \frac{1}{2} \bigl(
    \| d_k - (1-\tau_k) d_{k-1} + F(x_k) - F(y_k) \|^2 \\
    & - \| d_k - (1-\tau_k) d_{k-1}\|^2
    - \| F(x_k)-F(y_k)\|^2
    \bigr).
\end{aligned}
\end{equation}
It follows from the nonexpansiveness of the projection operator and the stepsize condition~\eqref{eq:line-search} that
\begin{equation*}
\scalebox{0.99}{$
\begin{aligned}
\| d_k - (1-\tau_k) d_{k-1} +  F(x_k) -F(y_k) \| & = \lambda_k^{-1}\| \tau_k u_{k-1} + (1-\tau_k)(y_{k-1} - \lambda_k d_{ k - 1 }) - y_k \| \\
& \ge \lambda_k^{-1} \|x_k - y_k\| \ge  \sigma^{-1}\| F(x_k) - F(y_k) \|,
\end{aligned}
$}
\end{equation*}
which, together with~\eqref{eq:polarization}, implies
\begin{equation}\label{eq:descent-estimate}
\langle F(x_k) - F(y_k) , d_k - ( 1 - \tau_k ) d_{k-1} \rangle \ge \frac{1-\sigma^2}{2\sigma^2} \| F( x_k ) -F( y_k )\|^2 - \frac{1}{2} \| d_k - ( 1 - \tau_k ) d_{k-1} \|^2.
\end{equation}
By substituting~\eqref{eq:monotonicity-gap} and~\eqref{eq:descent-estimate} into~\eqref{eq:lk-main-part} and rearranging terms, we obtain that
\begin{equation}\label{eq:lk-descent}
\begin{aligned}
    & \quad \langle u_{k-1}-y_{k}, d_k \rangle-(1-\tau_k)\langle u_{k-1}-y_{k-1},d_{k-1} \rangle\\ 
    & \ge \frac{\lambda_k(1-\sigma^2)}{2\tau_k\sigma^2}\|F(x_k)-F(y_k)\|^2 + \frac{\lambda_k}{2\tau_k}(\|d_k\|^2-(1-\tau_k)^2\|d_{k-1}\|^2).
\end{aligned}
\end{equation}
Now we are ready to establish the lower bound of $ \mathcal{L}_k - ( 1 - \tau_k ) \mathcal{L}_{ k - 1 } $. Combining~\eqref{eq:lk-difference} and~\eqref{eq:lk-descent}, we arrive at
\begin{equation*}
\begin{aligned}
    \mathcal{L}_k - ( 1 - \tau_k ) \mathcal{L}_{ k - 1 } \ge & \frac{\lambda_k(1-\sigma^2)}{2\tau_k\sigma^2}\|F(x_k)-F(y_k)\|^2 + \left(\frac{\lambda_k}{2\tau_k} - \rho\lambda_k - \frac{1-2\rho}{2}\Lambda_k \right) \| d_k \|^2 \\
    & + \left( \frac{(1-\tau_k)(1-2\rho)}{2} \Lambda_{k-1} - \frac{\lambda_k(1-\tau_k)^2}{2\tau_k} \right) \| d_{k-1} \|^2.
\end{aligned}
\end{equation*}
Since $\tau_k=\frac{\lambda_k}{(1-2\rho)\Lambda_{k-1}+\lambda_k}$, we have
\begin{equation*}
\frac{\lambda_k}{2\tau_k} - \rho\lambda_k - \frac{1-2\rho}{2}\Lambda_k=0 \text{  and  } \frac{\lambda_k(1-\tau_k)^2}{2\tau_k} - \frac{(1-\tau_k)(1-2\rho)}{2} \Lambda_{k-1} = 0.
\end{equation*}
This directly yields the conclusion.
\end{proof}

Based on \Cref{prop:lk-recursion}, we can establish the $\mathcal{O}(1/\Lambda_k)$ iteration complexity of the MAEG method.  

\begin{theorem}\label{thm:maeg-properties}
Suppose that Assumptions~\ref{asmp:existence} and~\ref{asmp:line-search} hold. Let the sequences $\{x_k\}$, $\{y_k\}$, and $\{u_k\}$ be generated by \Cref{alg:maeg}, and let the function $\mathcal{L}_k$ be defined as in~\eqref{eq:lk}. Then, the following statements hold:
\begin{enumerate}[(i)]
    \item For all $ k \ge 1 $, it holds that $ \mathcal{L}_k \ge 0$.
    \item For any solution $x_\star \in \operatorname{Sol}(F,B)$, the sequence $\{\|u_k-x_\star\|\}_{k \ge 1}$ is non-increasing. Specifically, for all $k \geq 1$, we have
    \begin{equation*}
    \|u_{k-1} - x_\star\|^2 - \|u_k - x_\star\|^2 \geq \big(\rho^2\lambda_k^2 + \rho(1-2\rho)\lambda_k\Lambda_k\big)\|d_k\|^2 \geq 0.
    \end{equation*}
    \item For any solution $x_\star \in \operatorname{Sol}(F,B)$ and all $ k \ge 1 $, we have
    \begin{equation*}
    \| \mathcal{R}_{\mathrm{nat}}^c (y_k) \| \leq \|d_k\| \leq \frac{2\|u_k-x_\star\|}{(1-2\rho)\Lambda_k} \leq \frac{2\|x_0-x_\star\|}{(1-2\rho)\Lambda_k}.
    \end{equation*}
\end{enumerate}
\end{theorem}

\begin{proof}
To prove statement \textnormal{(i)}, note that
$u_0=y_0=x_0$, $d_0=0$, and $\Lambda_0=0$, hence $\mathcal{L}_0=0$.
By Proposition~\ref{prop:lk-recursion}, for every $k\ge1$, $\mathcal{L}_k\ge (1-\tau_k)\mathcal{L}_{k-1}$. Since $1-\tau_k\ge0$, induction yields $\mathcal{L}_k\ge0$ for all
$k\ge1$.

Next, we prove statement (ii). For any $x_\star \in \operatorname{Sol}(F,B)$ and all $k \ge 1$, substituting $u_{k-1} = u_k + \lambda_k\rho d_k$ yields
\begin{equation}\label{eq:u-distance-expansion}
\|u_{k-1}-x_\star\|^2 = \|u_k-x_\star + \lambda_k\rho d_k\|^2 = \|u_k-x_\star\|^2 + 2\lambda_k\rho\langle u_k-x_\star, d_k \rangle + \lambda_k^2\rho^2\|d_k\|^2.
\end{equation}
Using statement (i) and the monotonicity of $(F+B)$, we have
\begin{equation}\label{eq:lk-solution-bound}
0 \leq \mathcal{L}_k = \langle u_k-y_k, d_k\rangle - \frac{1}{2}(1-2\rho)\Lambda_k\|d_k\|^2 \leq \langle u_k - x_\star, d_k\rangle - \frac{1}{2}(1-2\rho)\Lambda_k\|d_k\|^2,
\end{equation}
which implies $\langle u_k - x_\star, d_k \rangle \ge \frac{1}{2} ( 1 - 2 \rho ) \Lambda_k \| d_k \|^2$. Substituting this back into~\eqref{eq:u-distance-expansion}, we obtain
\begin{equation*}
\|u_{k-1} - x_\star\|^2 - \|u_k - x_\star\|^2 \geq \big(\rho^2\lambda_k^2 + \rho(1-2\rho)\lambda_k\Lambda_k\big)\|d_k\|^2 \geq 0, \quad \forall k \geq 1,
\end{equation*}
which completes the proof of statement (ii).

Finally, from~\eqref{eq:lk-solution-bound}, we deduce that
\begin{equation}\label{eq:lk-cauchy}
\frac{1}{2}(1-2\rho)\Lambda_k \|d_k\|^2 \leq \langle u_k-x_\star, d_k\rangle \leq \|u_k-x_\star\|\|d_k\|.
\end{equation}
Combining~\eqref{eq:lk-cauchy} with statement (ii) and Proposition~\ref{prop:reduced-gradient-norm} yields statement (iii).
\end{proof}

Among the results established above, \Cref{thm:maeg-properties}(ii) is of central importance, particularly for our subsequent analysis. In the next section, we will leverage this monotonicity property to propose a restart strategy tailored to the MAEG method and establish the convergence for problems involving non-Lipschitz continuous operators. Furthermore, a direct corollary of~\Cref{thm:maeg-properties}(iii) is the improved complexity rate for cases where $F$ is globally Lipschitz continuous with constant $L$. Specifically, by fixing the stepsizes to $\lambda_k = \sigma/L$, the MAEG method achieves an $\mathcal{O}(1/k)$ iteration complexity with respect to both the composite natural residual and the reduced gradient norm. This improves upon the $\mathcal{O}(1/\sqrt{k})$ rate reported in~\cite[Theorem 5]{nesterov2023high}.

Next, we shall establish the convergence results of \Cref{alg:maeg}. To this end, we estimate the upper bound of $ \mathcal{L}_k $ first and then introduce some propositions for the cases of $\rho = 0$ and $0<\rho<\frac{1}{2}$, respectively.

\begin{lemma}\label{lem:lk-upper-bound}
Suppose that Assumptions~\ref{asmp:existence} and~\ref{asmp:line-search} hold. Let the sequences $\{x_k\}$, $\{y_k\}$, and $\{u_k\}$ be generated by Algorithm~\ref{alg:maeg}, and let $\mathcal{L}_k$ be defined as in~\eqref{eq:lk}. Then, for all $x_\star \in \operatorname{Sol}(F,B)$ and $k \ge 1$, the following upper bound holds:
\begin{equation*}
    \mathcal{L}_k \leq \frac{ \| u_k - x_\star \|^2 }{ 2 ( 1 - 2 \rho ) \Lambda_k} \leq \frac{ \| x_0 - x_\star \|^2 }{ 2 ( 1 - 2 \rho ) \Lambda_k}.
\end{equation*}
\end{lemma}

\begin{proof}
The conclusion follows immediately by maximizing the quadratic form on the right-hand side of inequality~\eqref{eq:lk-solution-bound} with respect to $d_k$.
\end{proof}

\begin{proposition}\label{prop:rho-zero-summable}
Suppose that Assumptions~\ref{asmp:existence} and~\ref{asmp:line-search} hold. Let the sequences $\{x_k\}$, $\{y_k\}$ and $\{u_k\}$ be generated by \Cref{alg:maeg} with $\rho = 0$, then the following statements hold:
\begin{enumerate}[(i)]
\item For all $x_\star \in \operatorname{Sol}(F,B)$, we have $    \sum_{k=1}^{+\infty} \Lambda_k^2\|F(x_k) - F(y_k)\|^2\leq \frac{\sigma^2}{1 - \sigma^2}\| x_0 - x_\star\| ^2$.
\item For all $x_\star \in \operatorname{Sol}(F,B)$ and $ k \ge 1 $, we have $\| y_k - x_0 \|^2 \leq \left( \frac{2\sigma^2}{1-\sigma^2} + 8 \right)\| x_0 - x_\star \|^2$.
\end{enumerate}
\end{proposition}

\begin{proof}
When $\rho = 0$, we have $u_k = x_0$ and $\tau_k = \frac{\lambda_k}{\Lambda_k}$ for all $k \ge 1$. It follows from \Cref{prop:lk-recursion} that for all $k \ge 1$,
\begin{equation}\label{eq:lk-rho-zero}
    \mathcal{L}_k-\frac{\Lambda_{k-1}}{\Lambda_k}\mathcal{L}_{k-1}\ge \frac{(1 - \sigma^2)\Lambda_k}{2\sigma^2} \| F(x_k) - F(y_k) \|^2.
\end{equation}
Multiplying both sides of~\eqref{eq:lk-rho-zero} by $\Lambda_k$ and summing over $k$ from $1$ to $+\infty$ , we obtain that for all $x_\star\in \operatorname{Sol}(F,B)$,
\begin{equation*}
    \sum_{k = 1}^{+\infty}\Lambda_k^2 \| F(x_k) - F(y_k) \|^2 \leq \frac{2\sigma^2}{1 - \sigma^2} \lim_{k \to +\infty}\Lambda_k\mathcal{L}_k\leq \frac{\sigma^2}{1-\sigma^2}\|x_0-x_\star\|^2,
\end{equation*}
where the last inequality is due to \Cref{lem:lk-upper-bound}.

Now, we shall prove the statement (ii).
For all $ k \ge 1 $, the expression for $ y_k $ can be reformulated as 
\begin{equation*}
    y_k=\frac{\lambda_k}{\Lambda_k}x_0+\frac{\Lambda_{k-1}}{\Lambda_k}y_{k-1}-\lambda_k (F(x_k) - F(y_k) + d_k).
\end{equation*}
By leveraging the convexity of $ \| \cdot \|^2 $, we derive
\begin{equation}\label{eq:y-bound-recursion}
\begin{aligned}
    \|y_k-x_0\|^2 & =\left\|\tfrac{\Lambda_{k-1}}{\Lambda_k}(y_{k-1}-x_0)-\lambda_k (F(x_k)-F(y_k) + d_k)\right\|^2\\
    & \leq \tfrac{\Lambda_{k-1}}{\Lambda_k}\|y_{k-1}-x_0\|^2+\tfrac{\lambda_k}{\Lambda_k}\|\Lambda_k (F(x_k) - F(y_k) + d_k)\|^2.
\end{aligned}
\end{equation}
Using Theorem~\ref{thm:maeg-properties}(iii), we obtain that for all \(x_\star\in \operatorname{Sol}(F,B)\),
\begin{equation*}
\begin{aligned}
    \|\Lambda_k (F(x_k) - F(y_k) + d_k)\|^2 & \leq 2\Lambda_k^2\| F(x_k) - F(y_k)\|^2+2\Lambda_k^2\| d_k \|^2 \\
    &\leq 2\Lambda_k^2 \|F(x_k) - F(y_k)\|^2 +8 \|x_0-x_\star\|^2,
\end{aligned}
\end{equation*}
which, when substituted back into~\eqref{eq:y-bound-recursion}, yields
\begin{equation*}
    \Lambda_k\|y_k-x_0\|^2-\Lambda_{k-1}\|y_{k-1}-x_0\|^2\leq 2\lambda_k \Lambda_k^2 \| F(x_k) - F(y_k) \|^2 + 8\lambda_k \|x_0 - x_\star\|^2.
\end{equation*}
By using the statement (i), it follows that
\begin{equation*}
\begin{aligned}
    \|y_k-x_0\|^2 & \leq \frac{1}{\Lambda_k} \sum_{i=1}^k\left( 2 \lambda_i \Lambda_i^2 \| F(x_i) - F(y_i) \|^2 + 8 \lambda_i \| x_0 - x_\star \|^2 \right) \\
    &\leq 2 \sum_{i=1}^k \Lambda_i^2\| F(x_i) - F(y_i) \|^2+8 \|x_0 - x_\star\|^2 \leq \left( \frac{2\sigma^2}{1-\sigma^2} + 8 \right)\| x_0 - x_\star \|^2,
\end{aligned}
\end{equation*}
which completes the proof.
\end{proof}

\begin{proposition}\label{prop:yu-gap}
Suppose that Assumptions~\ref{asmp:existence} and~\ref{asmp:line-search} hold. Let the sequences $\{x_k\}$, $\{y_k\}$ and $\{u_k\}$ be generated by \Cref{alg:maeg} with $ 0< \rho < \frac{1}{2} $, if $\lambda_k\in (\underline{\lambda},M \lambda_1)$ for some $\underline{\lambda} > 0$ and $ M > 1 $, then for all $x_\star \in \operatorname{Sol}(F,B)$, we have
\begin{equation*}
    \| y_k - u_k \|^2 \leq \left( \tfrac{ 1 - \rho }{ \rho } + \tfrac{ \sigma^2 M }{ ( 1 - \sigma^2 ) ( 1 - 2 \rho )^2 } \right) \| x_0 - x_\star \|^2
    \text{ and } \lim_{k\to+\infty}\|y_k-u_k\|=0.
\end{equation*}
\end{proposition}

\begin{proof}
For any $x_\star \in \operatorname{Sol}(F,B)$, define the auxiliary function $\mathcal{M}_k$ as
\begin{equation*}
\mathcal{M}_k = \|y_k-u_k\|^2 + \frac{1-\rho}{\rho}\|u_k-x_\star\|^2.
\end{equation*}
Using the update rules for $\{y_k\}$ and $\{u_k\}$, we have
\begin{equation*}
y_k - u_k = (1-\tau_k)(y_{k-1}-u_{k-1}) - \lambda_k ( F(x_k) - F(y_k) + d_k ) + \rho\lambda_k d_k.
\end{equation*}
Rearranging terms implies
\begin{equation*}
(1-\tau_k)(y_{k-1}-u_{k-1}) = y_k - u_k + \lambda_k(1-\rho)d_k + \lambda_k(F(x_k)-F(y_k)).
\end{equation*}
Squaring both sides and applying the inequality $\|a+b\|^2 \ge (1-\tau_k)\|a\|^2 - \frac{1-\tau_k}{\tau_k}\|b\|^2$, we obtain
\begin{equation*}
(1-\tau_k)\|y_{k-1}-u_{k-1}\|^2 \ge \|y_k - u_k + \lambda_k(1-\rho)d_k\|^2 - \frac{\lambda_k^2}{\tau_k}\|F(x_k)-F(y_k)\|^2.
\end{equation*}
Expanding the norm squared term yields
\begin{equation*}
\begin{aligned}
\|y_k-u_k\|^2-\|y_{k-1}-u_{k-1}\|^2\leq &\frac{\lambda_k^2}{\tau_k}\|F(x_k)-F(y_k)\|^2-\tau_k\|y_{k-1}-u_{k-1}\|^2\\
& -\lambda_k^2(1-\rho)^2\|d_k\|^2 + 2\lambda_k(1-\rho)\langle u_k-y_k,d_k\rangle.
\end{aligned}
\end{equation*}
On the other hand, using $u_{k-1}=u_k+\rho\lambda_k d_k$, the squared distance to the solution satisfies
\begin{equation*}
\|u_k-x_\star\|^2-\|u_{k-1}-x_\star\|^2= -\rho^2\lambda_k^2\|d_k\|^2-2\rho\lambda_k\langle u_k-x_\star,d_k\rangle.
\end{equation*}
Substituting these estimates into the expression for $\mathcal{M}_k-\mathcal{M}_{k-1}$, we obtain
\begin{equation}\label{eq:mk-difference}
\begin{aligned}
\mathcal{M}_k - \mathcal{M}_{k-1}
&= \|y_k-u_k\|^2 - \|y_{k-1}-u_{k-1}\|^2 + \frac{1-\rho}{\rho}\left(\|u_k-x_\star\|^2 - \|u_{k-1}-x_\star\|^2\right) \\
&\leq \frac{\lambda_k^2}{\tau_k}\|F(x_k)-F(y_k)\|^2 - \tau_k\|y_{k-1}-u_{k-1}\|^2 - 2\lambda_k(1-\rho)\langle y_k-x_\star, d_k\rangle \\
&\quad - \lambda_k^2(1-\rho)\|d_k\|^2 \\
&\leq \frac{\lambda_k^2}{\tau_k}\|F(x_k)-F(y_k)\|^2 - \tau_k\|y_{k-1}-u_{k-1}\|^2,
\end{aligned}
\end{equation}
where the last inequality follows from the monotonicity property $\langle y_k - x_\star, d_k \rangle \ge 0$.
It follows from \Cref{prop:lk-recursion} and the bound $ \lambda_k \leq M \lambda_1$ that
\begin{equation*}
\begin{aligned}
\sum_{i=1}^{k}\frac{\lambda_i^2}{\tau_i}\|F(x_i)-F(y_i)\|^2 &\leq \sum_{i=1}^{k}\frac{2\sigma^2\lambda_i}{1-\sigma^2}\left(\mathcal{L}_i-(1-\tau_i)\mathcal{L}_{i-1}\right)\\
&\leq \frac{ 2 \sigma^2 M \lambda_1 }{ 1 - \sigma^2 } \left( \sum_{i=2}^{k} \tau_{i} \mathcal{L}_{i-1} + \mathcal{L}_k \right).
\end{aligned}
\end{equation*}
Using Lemma~\ref{lem:lk-upper-bound} and the condition $\tau_k \leq \frac{\lambda_k}{(1-2\rho)\Lambda_k}$, we bound the sum involving $\mathcal{L}$:
\begin{equation*}
\sum_{i=2}^{k} \tau_{i} \mathcal{L}_{i-1} + \mathcal{L}_k \leq \frac{ \| x_0 - x_\star \|^2 }{ 2 ( 1 - 2 \rho )^2} \left( \sum_{ i =2 }^{ k }\frac{ \lambda_{i} }{ \Lambda_{i-1} \Lambda_i } + \frac{1-2\rho}{\Lambda_k} \right) \leq \frac{ \| x_0 - x_\star \|^2 }{ 2 ( 1 - 2 \rho )^2 \lambda_1 }.
\end{equation*}
Therefore,
\begin{equation}\label{eq:mk-series-bound}
\sum_{ i = 1 }^{ k } \frac{ \lambda_i^2 }{ \tau_i } \| F(x_i) - F(y_i) \|^2 \leq \frac{ \sigma^2 M \| x_0 - x_\star \|^2}{ ( 1 - \sigma^2 ) ( 1 - 2 \rho )^2}.
\end{equation}
Summing~\eqref{eq:mk-difference} over $i=1,\dots,k$ and using~\eqref{eq:mk-series-bound} implies
\begin{equation*}
\begin{aligned}
    \| y_k - u_k \|^2 & \leq \mathcal{ M }_k \leq \mathcal{ M }_0 + \sum_{ i = 1 }^{ k } \frac{ \lambda_i^2 }{ \tau_i }\| F( x_i ) - F( y_i ) \|^2 \\
    & \leq \left( \frac{ 1 - \rho }{ \rho } + \frac{ \sigma^2 M }{ ( 1 - \sigma^2 ) ( 1 - 2 \rho )^2 } \right) \| x_0 - x_\star \|^2. 
\end{aligned}
\end{equation*}

Letting $k \to +\infty$ in~\eqref{eq:mk-series-bound}, we see that the sum converges. Consequently, from~\eqref{eq:mk-difference}, we deduce that $\lim_{k\to +\infty}\mathcal{M}_k$ exists. Since $\{\|u_k-x_\star\|\}$ converges (it is non-increasing and bounded below), it follows that $\lim_{k\to +\infty}\|y_k-u_k\|^2$ exists. Furthermore, rearranging~\eqref{eq:mk-difference}, we have
\begin{equation}\label{eq:tau-gap-summable}
\sum_{k=1}^{+\infty}\tau_k\|y_{k-1}-u_{k-1}\|^2 \leq \sum_{k=1}^{+\infty}(\mathcal{M}_{k-1} - \mathcal{M}_k) + \sum_{k=1}^{+\infty} \frac{\lambda_k^2}{\tau_k}\|F(x_k)-F(y_k)\|^2 < +\infty.
\end{equation}
Since $\tau_k \geq \frac{\lambda_k}{\Lambda_k} \geq \frac{\underline{\lambda}}{kM\lambda_1}$, we have $\sum_{k=1}^\infty \tau_k = +\infty$. If $\lim_{k\to+\infty}\|y_k-u_k\| \neq 0$, then there exists $\epsilon > 0$ such that $\|y_k-u_k\|^2 \ge \epsilon$ for sufficiently large $k$, which would imply $\sum \tau_k \|y_k-u_k\|^2 = +\infty$, contradicting~\eqref{eq:tau-gap-summable}. Thus, we conclude that $\lim_{k\to+\infty}\|y_k-u_k\|=0$.
\end{proof}

Now we are ready to establish the convergence results of \Cref{alg:maeg}. 
For $0 < \rho < \frac{1}{2}$, our approach is inspired by the proof of~\cite[Theorem 5]{yuan2024symplectic}.

\begin{theorem}\label{thm:maeg-convergence}
Suppose that Assumptions~\ref{asmp:existence} and~\ref{asmp:line-search} hold. Let the sequences $\{x_k\}$, $\{y_k\}$ and $\{u_k\}$ be generated by \Cref{alg:maeg}. If the stepsizes satisfy $\lambda_k\in (\underline{\lambda}, M \lambda_1)$ for some $\underline{ \lambda } > 0 $ and $M > 1$, then the following statements hold: 
\begin{enumerate}[(i)]
    \item If $\rho = 0$, then the sequence $\{y_k\}$ converges to $\Pi_{\operatorname{Sol}(F,B)}(x_0)$.
    \item If $0<\rho<\frac{1}{2}$, then both the sequences  $\{y_k\}$ and $\{u_k\}$ converge to the same point in $\operatorname{Sol}(F,B)$.
\end{enumerate}
\end{theorem}
\begin{proof}
To prove the statement (i), note that when $\rho = 0$, $u_k=x_0$ and $\tau_k = \frac{\lambda_k}{\Lambda_k}$ for all $k\ge 1$. From the definition of $d_k$ in~\eqref{eq:d-update}, we have
\begin{equation*}
    y_k=\tau_kx_0+(1-\tau_k)y_{k-1}-\lambda_k(F(x_k)-F(y_k))-\lambda_k d_k.
\end{equation*}
Since $d_k \in (F+B)(y_k)$ and $F+B$ is maximal monotone by \cite[Proposition A.1]{monteiro2011complexity}, we can equivalently write
\begin{equation*}
    y_k
    = J_{\lambda_k(F+B)}\!\left(\tau_k x_0 + (1-\tau_k)y_{k-1} - \lambda_k\bigl(F(x_k)-F(y_k)\bigr)\right).
\end{equation*}
Let $x_\star=\Pi_{\operatorname{Sol}(F,B)}(x_0)$. Since $x_\star \in \operatorname{Sol}(F,B)$, we have $0\in (F+B)(x_\star)$, which implies $x_\star = J_{\lambda_k( F + B )}(x_\star)$. Using the nonexpansiveness of the resolvent, it follows that 
\begin{equation}\label{eq:resolvent-nonexpansive}
\begin{aligned}
    \|y_k-x_\star\| & \leq \| \tau_k x_0 + ( 1 - \tau_k ) y_{ k-1 } - \lambda_k ( F( x_k ) - F( y_k ) ) - x_\star\| \\
    & = \| \tau_k ( x_0 - x_\star ) + ( 1 - \tau_k ) ( y_{ k - 1 } - x_\star) -\lambda_k ( F( x_k ) - F(y_k) )\|.
\end{aligned}
\end{equation}
When $k \ge 2$, we have $0<\tau_k<1$. By squaring both sides of inequality~\eqref{eq:resolvent-nonexpansive} and applying the Cauchy-Schwarz inequality, we obtain that for all $k \ge 2$,
\begin{equation}\label{eq:xu-recursion}
\scalebox{0.97}{$
\begin{aligned}
\|y_k-x_\star\|^2
&\le \frac{1}{1-\tau_k^2}
\bigl\|\tau_k(x_0-x_\star)+(1-\tau_k)(y_{k-1}-x_\star)\bigr\|^2
+ \frac{\lambda_k^2}{\tau_k^2}\|F(x_k)-F(y_k)\|^2 \\
&\le (1-\tau_k)\|y_{k-1}-x_\star\|^2
+ \tau_k\frac{\tau_k}{1-\tau_k^2}\|x_0-x_\star\|^2
+ \frac{2\tau_k}{1+\tau_k}\langle x_0-x_\star, y_{k-1}-x_\star\rangle \\
&\quad + \Lambda_k^2\|F(x_k)-F(y_k)\|^2 .
\end{aligned}
$}
\end{equation}
It follows from \Cref{prop:rho-zero-summable} (i) that
\begin{equation}\label{eq:xu-summable}
    \sum_{k=2}^{+\infty}\Lambda_k^2\| F(x_k) - F(y_k) \|^2 \leq \sum_{k=1}^{+\infty}\Lambda_k^2\| F(x_k) - F(y_k) \|^2 \leq \frac{\sigma^2}{1 - \sigma^2}\|x_0-x_\star\|^2 < +\infty. 
\end{equation}
We also have $ \prod_{k = 2}^{+\infty}(1-\tau_k) = \lim_{k\to +\infty}\frac{\lambda_1}{\Lambda_k}=0 $.
To apply \Cref{lem:xu}, we only need to demonstrate that 
\begin{equation}\label{eq:xu-limsup}
\limsup_{k \to +\infty} \left( \tfrac{\tau_k}{1-\tau_k^2}\|x_0-x_\star\|^2 + \tfrac{2}{1+\tau_k} \langle x_0 - x_\star, y_{k-1} - x_\star \rangle \right ) \leq 0.
\end{equation}

Let $\mathscr{Y}$ denote the set of accumulation points of the sequence $\{y_k\}$. By \Cref{prop:rho-zero-summable}(ii), the sequence $\{y_k\}$ is bounded, and thus the set $\mathscr{Y}$ is nonempty. For all $y_\star \in \mathscr{Y}$,  there exists a subsequence $\left\{y_{k_j}\right\}$ which converges to $y_\star$. According to \Cref{lem:resolvent-continuity}(iii) and \Cref{thm:maeg-properties}(iii), we can obtain that 
\begin{equation}\label{eq:accumulation-point}
    \| \mathcal{R}_{\mathrm{nat}}^c (y_\star) \| = \lim_{j \to +\infty} \| \mathcal{R}_{\mathrm{nat}}^c ( y_{k_j} ) \| \leq \lim_{j \to +\infty} \frac{2 \|x_0-x_\star\|}{\Lambda_{k_j}} = 0,
\end{equation}
It follows that $ \| \mathcal{R}_{\mathrm{nat}}^c (y_\star)  \| = 0 $, which implies $ y_\star \in \operatorname{Sol}(F,B) $. Due to the property of the projection operator, we have $\langle x_0 - x_\star , y_\star - x_\star  \rangle \leq 0$. Therefore, we obtain that
\begin{equation*}
    \limsup_{ k \to +\infty} \langle x_0 - x_\star, y_{k-1} - x_\star  \rangle \leq \sup_{y_\star \in \mathscr{Y}} \langle x_0 - x_\star , y_\star - x_\star\rangle \leq 0,
\end{equation*}
which, combining with $\tau_k\to 0$, implies~\eqref{eq:xu-limsup} holds. Now we can obtain that $ y_k \to x_\star $ as $k \to +\infty$ using \Cref{lem:xu}.

Now, we begin to prove the statement (ii) where $0<\rho<\frac{1}{2}$. By \Cref{thm:maeg-properties}(ii), we have that the sequence $\{u_k\}$ is bounded. From \Cref{prop:yu-gap}, we know $ \| u_k - y_k \| \to 0$ as $k \to + \infty $, which implies the sequence $\{y_k\}$ is also bounded. Take a convergent subsequence $\{y_{k_j}\}$ and let $y_{k_j}\to \bar{x}$. Similar to~\eqref{eq:accumulation-point}, we can obtain $\bar{x}\in\operatorname{Sol}(F,B)$. On the other hand, since $\{\|u_k-\bar{x}\|\}$ is non-increasing, we have
\begin{equation*}
    \lim_{k\to +\infty}\|u_k-\bar{x}\|=\lim_{j\to +\infty}\|u_{k_j}-\bar{x}\|\leq \lim_{j\to +\infty}\|u_{k_j}-y_{k_j}\| + \lim_{j\to +\infty}\|y_{k_j}-\bar{x}\| =0,
\end{equation*}
which implies $u_k\to \bar{x}$. Using \Cref{prop:yu-gap} again, we have $y_k\to\bar{x}$. Therefore,  both the sequence  $\{y_k\}$ and $\{u_k\}$ converge to the same point $ \bar{x} \in \operatorname{Sol}(F,B)$.
\end{proof}

\begin{remark}
When the operator $ F(\cdot) $ is Lipschitz continuous with constant $L$ and the stepsizes are set to $\lambda_k = \frac{\sigma}{L}$ for all $ k \ge 1 $, if $ 0 < \rho < \frac{1}{2} $, then $\{u_k\}$ converges to some $\bar{x}\in\operatorname{Sol}(F,B)$. According to Theorem~\ref{thm:maeg-properties}(iii), we can obtain that for all $ k \ge 1 $,
\begin{equation*}
\| \mathcal{R}_{\mathrm{nat}}^c( y_k ) \| \leq \|d_k\| \leq \frac{ 2 L \|u_k - \bar{x} \| }{( 1 - 2 \rho) \sigma k},
\end{equation*}
which, together with the fact that $\lim_{k \to +\infty} \| u_k - \bar{x} \| = 0 $, implies the $o( 1 / k )$ asymptotic convergence rate.
\end{remark}

\section{Restart Strategy for Non-Lipschitz Continuous Monotone Inclusions}\label{sec:maeg-r}
In this section, we introduce a restart strategy specifically designed for the MAEG method. The resulting algorithm, termed MAEG-R, converges to a solution of~\eqref{eq:mi} without requiring any Lipschitz assumption. Moreover, when $F(\cdot)$ is Lipschitz continuous, MAEG-R preserves the corresponding complexity guarantees of MAEG after finitely many restarts. The details of the algorithm are presented in~\Cref{alg:maeg-r}.
\begin{algorithm}[!htbp]
	\caption{Moving-Anchored Extra-Gradient Method with Restarts (MAEG-R)}
	\label{alg:maeg-r}
	\begin{algorithmic}[1]
    \STATE \textbf{Input:} $x_0\in\operatorname{cl}(\operatorname{dom} B)$, a nonempty, closed, convex set $X$ satisfying $\operatorname{cl}(\operatorname{dom} B)\subseteq X\subseteq \operatorname{dom} F$, a positive sequence $\{\varepsilon_r\}$ with $\varepsilon_r\to0$, constants $0<\underline{\alpha}_1\le \overline{\alpha}_1$, and parameters $\rho\in(0,\frac{1}{2})$, $\beta,\sigma,m\in(0,1)$, $\kappa\in [0,2]$, $M>1$.

    \STATE \textbf{Initialization:} Set $u_{1,0}=x_{1,0}=y_{1,0}=x_0$, $d_{1,0}=0$, $\Lambda_{r,0}=0$, $(\underline{\lambda}_r,\overline{\lambda}_r)=(0,+\infty)$ for all $r\in\mathbb{N}$, and $r=k=1$.

    \STATE \textbf{Step 1 (Termination Check):} If $\|y_{r,k-1}-J_B(y_{r,k-1}-F(y_{r,k-1}))\|=0$, terminate the algorithm.

    \STATE \textbf{Step 2 (MAEG Step):} Choose a trial stepsize $\alpha_{r,k}>0$. If $k=1$, require $\underline{\alpha}_1\le \alpha_{r,1}\le \overline{\alpha}_1$; if $k\ge2$, require $\alpha_{r,k}\ge \lambda_{r,k-1}$. Set
    \[
        \lambda_{r,k}=\min\{\alpha_{r,k},\overline{\lambda}_r\}\beta^j,
    \]
    $j$ is the smallest non-negative integer such that
    \begin{equation}\label{eq:maeg-r-linesearch-test}
        \lambda_{r,k}<\underline{\lambda}_r
        \quad\text{or}\quad
        \lambda_{r,k}\|F(x_{r,k})-F(y_{r,k})\|\le \sigma \|x_{r,k}-y_{r,k}\|,
    \end{equation}
    where $(x_{r,k},y_{r,k})$ in~\eqref{eq:maeg-r-linesearch-test} is the trial pair generated by~\eqref{eq:x-update}--\eqref{eq:y-update}, with $(y_{k-1},u_{k-1},d_{k-1},\lambda_k,\Lambda_{k-1})$ replaced by $(y_{r,k-1},u_{r,k-1},d_{r,k-1},\lambda_{r,k},\Lambda_{r,k-1})$. Then compute $(x_{r,k},y_{r,k},d_{r,k},u_{r,k},\Lambda_{r,k})$ by applying~\eqref{eq:x-update}--\eqref{eq:u-update} under the same replacements.

    \STATE \textbf{Step 3 (Stepsize Bounds):} If $k=1$, set
    \[
        (\underline{\lambda}_r,\overline{\lambda}_r)=\bigl(\min\{m\lambda_{r,1},\varepsilon_r\},\,M\lambda_{r,1}\bigr).
    \]

    \STATE \textbf{Step 4 (Restart Check):} If $\lambda_{r,k}<\underline{\lambda}_r$, reject the obtained iterate $(x_{r,k},y_{r,k},d_{r,k},u_{r,k})$ and compute
    \begin{equation}\label{eq:restart-u}
    u_{r+1,0}=\Pi_{\operatorname{cl}(\operatorname{dom} B)}\left(u_{r,k-1}-\kappa \frac{\langle u_{r,k-1}-y_{r,k-1},d_{r,k-1}\rangle}{\|d_{r,k-1}\|^2}d_{r,k-1}\right),
    \end{equation}
    set $x_{r+1,0}=y_{r+1,0}=u_{r+1,0}$, $d_{r+1,0}=0$, update $(r,k)\gets (r+1,1)$, and return to \textbf{Step 1}.

    \STATE \textbf{Step 5:} Set $k\gets k+1$ and return to \textbf{Step 1}.

    \STATE \textbf{Output:} Iteration sequence $\{(x_{r,k},y_{r,k},u_{r,k})\}$.
	\end{algorithmic}
\end{algorithm}

It remains to show that, for $k=1$ and any $r$, \textbf{Step 2} in Algorithm~\ref{alg:maeg-r} is well-defined. Rather than addressing this in isolation, we establish a stronger result that facilitates further analysis. To this end, we first recall the Tube Lemma from general topology.

\begin{lemma}[Tube Lemma, Lemma 26.8 in~\cite{munkres2000topology}]
\label{lem:tube-lemma}
Let $X$ and $Y$ be topological spaces, where $X$ is compact. Let $y_0 \in Y$, and let $O$ be an open set of $X \times Y$ containing the slice $X \times \{y_0\}$. Then there exists a neighborhood $V$ of $y_0$ in $Y$ such that $O$ contains the tube $X \times V$.
\end{lemma}

For $k=1$, since $x_{r,1}=x_{r,0}$ and $y_{r,1}=J_{\lambda_{r,1}B}(x_{r,0}-\lambda_{r,1}F(x_{r,0}))$, the line-search step is well defined by the stronger result below.

\begin{proposition}\label{prop:stepsize-lower-bound}
Let $\sigma \in (0,1)$. Suppose $S \subseteq \operatorname{cl}(\operatorname{dom} B) \setminus \operatorname{Sol}(F,B)$ is a compact set and that $F(\cdot)$ is continuous on $\operatorname{cl}(\operatorname{dom} B)$. Then, there exists $\delta > 0$ such that
\begin{equation*}
    \lambda\|F(J_{\lambda B}(x-\lambda F(x)))-F(x)\|\leq \sigma \|J_{\lambda B}(x-\lambda F(x))-x\|, \quad \forall (x,\lambda)\in S\times (0,\delta].
\end{equation*}
\end{proposition}

\begin{proof}
Since $S \subseteq \operatorname{cl}(\operatorname{dom} B) \setminus \operatorname{Sol}(F,B)$ is compact and $F(\cdot)$ is continuous on $\operatorname{cl}(\operatorname{dom} B)$, the residual function is bounded away from zero on $S$. Specifically, there exists a constant $\delta_0 > 0$ such that, for all $x \in S$, $\|J_{B}(x-F(x))-x\| \geq \delta_0 > 0 $.
It follows from Lemma~\ref{lem:nonincreasing-resolvent} that the quotient $\frac{1}{\lambda}\|J_{\lambda B}(x-\lambda F(x))-x\|$ is non-increasing with respect to $\lambda$, which implies
\begin{equation}\label{eq:resolvent-lower-bound}
    \frac{1}{\lambda}\|J_{\lambda B}(x-\lambda F(x))-x\| \geq \|J_{B}(x-F(x))-x\| \geq \delta_0, \quad \forall (x,\lambda) \in S \times (0,1].
\end{equation}
On the other hand, Lemma~\ref{lem:resolvent-continuity}(iii) ensures that the function $\varphi(\cdot,\cdot)$, defined in~\eqref{eq:varphi}, is continuous on $S\times [0,1]$. We define a new continuous function $\Phi: \operatorname{cl}(\operatorname{dom}B) \times [0,1] \to \mathbb{R}$ by
\begin{equation*}
    \Phi(x,\lambda) = \|F(\varphi(x-\lambda F(x),\lambda))-F(x)\| = 
    \begin{cases}
        \|F(J_{\lambda B}(x-\lambda F(x))) - F(x)\|, & \text{if } \lambda > 0, \\
        0, & \text{if } \lambda = 0.
    \end{cases}
\end{equation*}
Let $O$ be the set defined by $O = \left\{(x,\lambda) \in \operatorname{cl}(\operatorname{dom}B) \times [0,1] \mid \Phi(x,\lambda) < \sigma \delta_0 \right\}$. Since $\Phi$ is continuous, $O$ is an open set in the subspace topology of $\operatorname{cl}(\operatorname{dom}B)\times [0,1]$. Furthermore, since $\Phi(x,0) = 0 < \sigma \delta_0$ for all $x \in S$, the slice $S \times \{0\}$ is contained in $O$. By the Tube Lemma (Lemma~\ref{lem:tube-lemma}), there exists $\delta \in (0,1)$ such that $S \times [0,\delta] \subseteq O$. This implies that
\begin{equation}\label{eq:forward-difference-bound}
    \|F(J_{\lambda B}(x-\lambda F(x))) - F(x)\| < \sigma \delta_0, \quad \forall (x,\lambda) \in S \times (0,\delta].
\end{equation}
Combining~\eqref{eq:resolvent-lower-bound} and~\eqref{eq:forward-difference-bound}, we obtain that
\begin{equation*}
    \lambda \|F(J_{\lambda B}(x-\lambda F(x)))-F(x)\| < \lambda \sigma \delta_0 \leq \sigma \|J_{\lambda B}(x-\lambda F(x))-x\|
\end{equation*}
holds for all $(x,\lambda) \in S \times (0,\delta]$. This completes the proof.
\end{proof}

Now, we are ready to present the convergence result for \Cref{alg:maeg-r}. Here, we no longer assume Assumption~\ref{asmp:line-search} holds, since we have proven it holds naturally when $k=1$, and when $k>1$, the well-definedness of Algorithm~\ref{alg:maeg-r} is ensured by the restart condition. To avoid ambiguity, throughout this section, the double-indexed iterates generated by Algorithm~\ref{alg:maeg-r} are understood in the order in which they are generated.

\begin{theorem}\label{thm:maeg-r-convergence}
Suppose Assumption~\ref{asmp:existence} holds. Let $\{(x_{r,k},y_{r,k},u_{r,k})\}$ be generated by Algorithm~\ref{alg:maeg-r}. Then $\{y_{r,k}\}$ and $\{u_{r,k}\}$ both converge to the same point in $\operatorname{Sol}(F,B)$.
\end{theorem}

\begin{proof}
We consider two cases regarding the restart mechanism.\\
\noindent \textbf{Case (i): Algorithm~\ref{alg:maeg-r} restarts only a finite number of times.}

Suppose Algorithm~\ref{alg:maeg-r} restarts $R-1$ times. Then for all sufficiently large $k$, the algorithm remains in the $R$-th epoch. Hence, the stepsize satisfies $\underline{\lambda}_R \leq \lambda_{R,k} \leq \overline{\lambda}_R$. By Theorem~\ref{thm:maeg-convergence}, the sequences $\{u_{R,k}\}$ and $\{y_{R,k}\}$ converge to the same point $x_\star \in \operatorname{Sol}(F,B)$. This completes the proof for the first case.

\medskip
\noindent \textbf{Case (ii): Algorithm~\ref{alg:maeg-r} restarts infinitely many times.}

Suppose for each epoch $r \in \mathbb{N}$, Algorithm~\ref{alg:maeg-r} accepts $k_r\ge 1$ updates before restarting. At such a restart, $d_{r,k_r}\neq0$; otherwise Proposition~\ref{prop:reduced-gradient-norm} gives $\mathcal{R}_{\mathrm{nat}}^c(y_{r,k_r})=0$, and the algorithm would stop. Moreover, by Theorem~\ref{thm:maeg-properties}(i), $\langle u_{r,k_r}-y_{r,k_r},d_{r,k_r}\rangle
\ge \frac{1}{2}(1-2\rho)\Lambda_{r,k_r}\|d_{r,k_r}\|^2>0$.
The transition to the next epoch is defined by
\begin{equation*}
u_{r+1,0} = \Pi_{\operatorname{cl}(\operatorname{dom} B)}\left(u_{r,k_r} - \kappa \gamma_{r,k_r} d_{r,k_r}\right),\quad \gamma_{r,k_r} = \langle u_{r,k_r} - y_{r,k_r},d_{r,k_r}\rangle /\|d_{r,k_r}\|^2>0.
\end{equation*}
Using the nonexpansiveness of $\Pi_{\operatorname{cl}(\operatorname{dom} B)}(\cdot)$, for any $x_\star \in \operatorname{Sol}(F,B)$, we have
\begin{equation*}
\begin{aligned}
\|u_{r+1,0}-x_\star\|^2 
&= \left\| \Pi_{\operatorname{cl}(\operatorname{dom} B)}\left(u_{r,k_r} - \kappa\gamma_{r,k_r}d_{r,k_r}\right) - x_\star \right\|^2 \\
&\leq \|u_{r,k_r}- x_\star - \kappa \gamma_{r,k_r}d_{r,k_r}\|^2 \\
&= \|u_{r,k_r}-x_\star\|^2 - 2\kappa \gamma_{r,k_r}\langle u_{r,k_r} - x_\star, d_{r,k_r}\rangle + \kappa^2 \gamma_{r,k_r}^2 \|d_{r,k_r}\|^2 \\
&\leq \|u_{r,k_r}-x_\star\|^2 - (2\kappa-\kappa^2) \langle u_{r,k_r} - y_{r,k_r}, d_{r,k_r}\rangle^2 / \|d_{r,k_r}\|^2,
\end{aligned}
\end{equation*}
where the last inequality is derived from the fact that $\langle u_{r,k_r} - x_\star, d_{r,k_r}\rangle \geq \langle u_{r,k_r} - y_{r,k_r}, d_{r,k_r}\rangle$ and the definition of $\gamma_{r,k_r}$. Since $0\leq \kappa \leq 2$, we obtain $\|u_{r+1,0}-x_\star\|\leq \|u_{r,k_r}-x_\star\|$. Furthermore, from Theorem~\ref{thm:maeg-properties}(ii), the sequence of distances $\{\|u_{r,k}-x_\star\|\}_{k \geq 0}$ is non-increasing within an epoch, which implies $\|u_{r+1,0}-x_\star\| \leq \|u_{r,k_r}-x_\star\| \leq \|u_{r,1}-x_\star\|$.
    
On the other hand, applying Theorem~\ref{thm:maeg-properties}(ii) with $k=1$, we have:
\begin{equation}\label{eq:restart-u-bound}
    \|u_{r,0}-x_\star\|^2 
    \ge \|u_{r,1}-x_\star\|^2 + (\rho-\rho^2)\lambda_{r,1}^2\|d_{r,1}\|^2 
    \ge \|u_{r+1,0}-x_\star\|^2 + (\rho-\rho^2)\lambda_{r,1}^2\|d_{r,1}\|^2.
\end{equation}
Summing~\eqref{eq:restart-u-bound} over $r$ from $1$ to $+\infty$, we obtain
\begin{equation}\label{eq:restart-series-upper-bound}
    \sum_{r=1}^{+\infty} \lambda_{r,1}^2\|d_{r,1}\|^2 \leq \frac{1}{\rho(1-\rho)} \|u_{1,0}-x_\star\|^2 < +\infty.
\end{equation}

Next, we estimate a lower bound for $\sum_{r=1}^{+\infty} \lambda_{r,1}^2\|d_{r,1}\|^2$. When $k=1$, the stepsize condition is equivalent to 
\begin{equation}\label{eq:restart-d-lower-bound-1}
    \lambda_{r,1}\|F(y_{r,1})-F(x_{r,1})\|\leq \sigma \|J_{\lambda_{r,1}B}(x_{r,1}-\lambda_{r,1}F(x_{r,1}))-x_{r,1}\|.
\end{equation}
Recall that from the definition of $d_{r,1}$, we have
\begin{equation}\label{eq:restart-d-lower-bound-2}
    F(y_{r,1})-F(x_{r,1})=d_{r,1}-\frac{1}{\lambda_{r,1}}(x_{r,1}-J_{\lambda_{r,1}B}(x_{r,1}-\lambda_{r,1}F(x_{r,1})))
\end{equation}
Combining~\eqref{eq:restart-d-lower-bound-1},~\eqref{eq:restart-d-lower-bound-2} and the triangle inequality yields
\begin{equation*}
    \|d_{r,1}\| \ge \frac{1-\sigma}{\lambda_{r,1}}\|x_{r,1}-J_{\lambda_{r,1}B}(x_{r,1}-\lambda_{r,1}F(x_{r,1}))\|.    
\end{equation*}
Since $x_{r,1}=u_{r,0}$ and $\lambda_{r,1} \leq \overline{\alpha}_1$, using \Cref{lem:nonincreasing-resolvent}, we obtain 
\begin{equation*}
    \|d_{r,1}\| \ge \frac{1-\sigma}{\overline{\alpha}_1}\|u_{r,0}-J_{\overline{\alpha}_1B}(u_{r,0}-\overline{\alpha}_1 F(u_{r,0}))\|.    
\end{equation*}

We now prove the main result by contradiction. From~\eqref{eq:restart-u-bound}, the sequence $\{u_{r,0}\}$ is bounded. Thus, there exists a subsequence $\{u_{r_j,0}\}$ converging to a point $\hat{x}$ as $j \to +\infty$. Suppose $\hat{x} \notin \operatorname{Sol}(F,B)$. Let $s = \operatorname{dist}(\hat{x}, \operatorname{Sol}(F,B)) > 0$ and consider the compact set $S = \mathbb{B}(\hat{x}, s/2)\cap \operatorname{cl}(\operatorname{dom} B)$. By Proposition~\ref{prop:stepsize-lower-bound}, there exists $\delta > 0$ such that $\lambda\|F(J_{\lambda B}(x-\lambda F(x)))-F(x)\| \leq \sigma \|J_{\lambda B}(x-\lambda F(x))-x\|$, $\forall \lambda \in (0,\delta]$, $\forall x \in S$.
Furthermore, since $S$ is compact and $S\cap\operatorname{Sol}(F,B) = \emptyset$, there exists a constant $\delta_0 > 0$ such that $\|x - J_{\overline{\alpha}_1 B}(x - \overline{\alpha}_1 F(x))\| \geq \delta_0 > 0$, $\forall x\in S$.
For sufficiently large $j$, say $j \geq K$, we have $u_{r_j,0} \in S$. This implies both $\lambda_{r_j,1}\ge \delta_1 :=\min(\beta\delta,\underline{\alpha}_1)>0$ and $\|u_{r_j,0}-J_{\overline{\alpha}_1 B}(u_{r_j,0}-\overline{\alpha}_1 F(u_{r_j,0}))\|\geq \delta_0>0$.
Therefore, we have
\begin{equation*}
\sum_{r=1}^{+\infty} \lambda_{r,1}^2\|d_{r,1}\|^2\ge \sum_{j=K}^{+\infty} \lambda_{r_j,1}^2\|d_{r_j,1}\|^2\ge \sum_{j=K}^{+\infty} (\frac{1-\sigma}{\overline{\alpha}_1}\delta_0\delta_1)^2 = +\infty,
\end{equation*}
which contradicts~\eqref{eq:restart-series-upper-bound}. Consequently, we have $\hat{x} \in \operatorname{Sol}(F,B)$. Since the distance sequence $\|u_{r,k} - \hat{x}\|$ is monotonically non-increasing and the subsequence converges to $\hat{x}$, the whole sequence $\{u_{r,k}\}$ converges to $\hat{x}$. It follows from Proposition~\ref{prop:yu-gap} that
\begin{equation*}
    \| y_{r,k} - u_{r,k} \|^2 \leq \left( \frac{ 1 - \rho }{ \rho } + \frac{ \sigma^2 M }{ ( 1 - \sigma^2 ) ( 1 - 2 \rho )^2 } \right) \| u_{r,0} - \hat{x} \|^2.
\end{equation*}
Since $\|u_{r,0}-\hat{x}\| \to 0$ and $u_{r,k} \to \hat{x}$ as $r \to +\infty$, we conclude that $y_{r,k}\to\hat{x}$.
\end{proof}
\begin{remark}
When $F(\cdot)$ is globally Lipschitz continuous with constant $L$, since $\underline{\lambda}_r\le\varepsilon_r\to 0$ as $r\to +\infty$, Algorithm~\ref{alg:maeg-r} restarts only a finite number of times, so the complexity results of Algorithm~\ref{alg:maeg} still hold.
\end{remark}

\begin{remark}\label{rmk:convergence}
    The convergence result above is insensitive to the specific rule used to trigger a restart. In particular, additional restart criteria can be incorporated into Algorithm~\ref{alg:maeg-r}, provided that each restart is performed after an accepted MAEG update and the new anchor is updated according to~\eqref{eq:restart-u}. The same proof as that of Theorem~\ref{thm:maeg-r-convergence} then applies.
\end{remark}

\section{Numerical Experiments}
In this section, we report some numerical experiments to evaluate the performance of the proposed MAEG method on a diverse collection of problem instances. The test set is designed to cover a range of operator classes, including cases where the operator $F(\cdot)$ is globally Lipschitz continuous, locally Lipschitz continuous, or merely continuous.

All algorithms are implemented in Julia~1.12.3, and all experiments are executed on an NVIDIA A100-SXM4-80GB GPU. Our implementations are written to perform the iterative updates with zero GPU memory allocations, and they leverage kernel fusion to reduce kernel launch overhead and improve overall computational efficiency.

To validate the efficiency of the proposed algorithm, we compare the following methods.
\begin{enumerate}
    \item[(i)] \textbf{MAEG-u / MAEG-y:}
    the proposed Moving-Anchored Extra-Gradient method with line-search parameters
    $(\beta,\sigma)=(0.7,0.99)$ and trial stepsize
    $\alpha_{r,k}=\min\{1.02\lambda_{r,k-1},1000\lambda_{r,1}\}$. MAEG-u uses $\rho=0.2$ and restarts from the moving anchor point according to~\eqref{eq:restart-u} with $\kappa=2$,
    whereas MAEG-y uses $\rho=0$ and restarts from the current iterate by setting
    $y_{r+1,0}=u_{r+1,0}=y_{r,k}$.
    These two values of $\rho$ are used as the default choices because they perform well for the corresponding restart strategies in our preliminary tests.

    \item[(ii)] \textbf{CFEG:}
    the composite Fast Extra-Gradient method~\cite{lee2021fast,cai2024accelerated}
    with fixed stepsize $0.99/L$, where $L$ is the global Lipschitz constant of $F(\cdot)$.
    We also add a projection step onto $X$ in the same way as in MAEG.
    This baseline is used only in the globally Lipschitz experiments.

    \item[(iii)] \textbf{MFBS:} 
    the Modified Forward-Backward Splitting method~\cite{tseng2000modified}.
    Following the practice in our experiments, we use the same line-search strategy as in MAEG-u, which improves empirical performance and is consistent with the analysis in~\cite{tseng2000modified}.
\end{enumerate}

To further enhance the performance of \textbf{MAEG-u / MAEG-y}, we incorporate adaptive restart mechanisms based on the behavior of $\|d_k\|$. Following ideas similar to those in~\cite{chen2025hpr}, a restart is triggered whenever one of the following conditions is met:
\begin{enumerate}[(i)]
    \item Sufficient decay of $\| d_{r,k} \|$: $ \| d_{r,k}\| \leq \alpha_1 \| d_{r,1} \| $;
    
    \item Necessary decay + no local progress of $\| d_{r,k} \|$: $ \| d_{r,k} \| \leq \alpha_2 \| d_{r,1} \|$ and $ \| d_{r,k+1} \| > \| d_{r,k} \|$;
    \item 
    Long inner loop: $k \geq \alpha_3 K$, where $K$ is the cumulative number of accepted iterations, including the current inner loop, and is not reset after restarts.
\end{enumerate}
The first two conditions can be viewed as restart criteria with respect
to $y_{r,k}$, since $d_{r,k}\in(F+B)(y_{r,k})$. For MAEG-u, where $\rho>0$, they can also be interpreted with respect to $u_{r,k}$ through
$d_{r,k}=(u_{r,k-1}-u_{r,k})/(\rho\lambda_{r,k})$.
Throughout the experiments, we set
$(\alpha_1,\alpha_2,\alpha_3)=(0.1,0.6,0.2)$. The same restart strategy is also applied to CFEG for consistency.

We adopt the stopping criterion based on the relative composite natural residual, defined as
\begin{equation*}
   \widetilde{\mathcal{R}}^c_{\mathrm{nat}}(x)
   = \frac{\| x - J_B( x - F(x) ) \|}
          {1 + \| x \|_{\infty} + \| F(x) \|_{\infty}} .
\end{equation*}
All algorithms are terminated once the residual satisfies
$\widetilde{\mathcal{R}}^c_{\mathrm{nat}}(x) < \epsilon$,
with a fixed tolerance $\epsilon = 10^{-6}$, or when a maximum wall-clock time of $3600$ seconds is reached. The residual is evaluated every $100$ iterations.

\subsection{Bilinear Matrix Games}

We consider two-player zero-sum matrix games defined by a payoff matrix \( A \in \mathbb{R}^{n \times n} \). Player I chooses a mixed strategy \( x \in \Delta_n \), and Player II chooses \( y \in \Delta_n \), where the \( k \)-dimensional probability simplex is defined as $\Delta_k = \left\{ (x_1, \dots, x_k) \in \mathbb{R}_+^k \mid \sum_{i=1}^k x_i = 1 \right\}$. The game can be modeled as the classic convex-concave minimax problem:
\begin{equation*}
    \min_{x \in \Delta_n} \max_{y \in \Delta_n} x^\top A y.
\end{equation*}
This problem can be further reformulated as a monotone inclusion problem~\eqref{eq:mi} over the compact convex set \( K = \Delta_n \times \Delta_n \), with monotone operator $F(x, y) = (A y, -A^\top x)$ and $ B(x, y) = (N_{\Delta_n}(x), N_{\Delta_n}(y))$, 
where \( N_{\Delta_k} \) denotes the normal cone to \( \Delta_k \). Existence of a solution is guaranteed by Nash’s theorem~\cite{nash1951non}. Note that $F(x,y)$ is globally Lipschitz continuous with constant
$L=\sqrt{\lambda_1(A^\top A)}$, where $\lambda_1(A^\top A)$ denotes
the largest eigenvalue of the symmetric matrix
$A^\top A$. In the experiments, we estimate $\lambda_1(A^\top A)$
by the power method, using products with $A$ and $A^\top$.

To evaluate $F(x,y)$, we employ the \texttt{CUSPARSE\_SPMV\_CSR\_ALG2} routine whenever $A$ is sparse, which guarantees deterministic results across GPU executions.
The resolvent of $B$ corresponds to the Euclidean projection onto the simplex $\Delta_n$. For any $v\in\mathbb{R}^n$, this projection can be computed by:
\begin{equation} \label{eq:simplex-root}
    (\Pi_{\Delta_n}(v))_i = \max(v_i-\tau,0), \text{ where }
    \phi(\tau) \coloneqq 1 - \sum_{i=1}^n \max(v_i - \tau, 0) = 0.
\end{equation}
While the univariate equation $\phi(\tau) = 0$ can be solved efficiently on the CPU by sorting/pivot-based algorithm~\cite{duchi2008efficient, condat2016fast}, these algorithms are ill-suited for the SIMT (Single Instruction, Multiple Threads) architecture of GPUs. The recursive pivot selection induces high branch divergence within warps and requires non-coalesced memory access patterns, severely degrading performance.

Instead, we exploit the fact that $\phi(\tau)$ is a monotone, piecewise-linear function and solve~\eqref{eq:simplex-root} using a semismooth Newton method, summarized in \Cref{alg:ssn-simplex}.

\begin{algorithm}[!htbp]
    \caption{Semismooth Newton Method for Simplex Projection}
    \label{alg:ssn-simplex}
    \begin{algorithmic}[1]
    
    \STATE \textbf{Input:} Vector $v \in \mathbb{R}^n$ and tolerance $\epsilon > 0$.

    \STATE \textbf{Initialization:} Set the initial guess $\tau_0 = \left(\sum_{i=1}^n v_i - 1\right)/n$ and initialize $k = 0$.

    \STATE \textbf{Step 1:} Compute $\phi(\tau_k)$ and the generalized Jacobian $J_k$
    \begin{equation*}
        \phi(\tau_k) = 1 - \sum_{i=1}^n \max(0, v_i - \tau_k)  \quad \text{and} \quad J_k = \sum_{i=1}^n \mathbb{I}(v_i > \tau_k).
    \end{equation*}
    
    \STATE \textbf{Step 2:} If $|\phi(\tau_k)| < \epsilon$, proceed to \textbf{Step 3}. Otherwise, compute the update
    \begin{equation*}
        \tau_{k+1} = \tau_k - J_k^{-1}\phi(\tau_k),
    \end{equation*}
    increment $k \gets k+1$ and return to \textbf{Step 1}.

    \STATE \textbf{Step 3:} Compute the final projection vector $w$ by
    \begin{equation*}
        w_i = \max(0, v_i - \tau_k), \quad i = 1, \dots, n.
    \end{equation*}

    \STATE \textbf{Output:} Projected vector $w$.
    
    \end{algorithmic}
\end{algorithm}

This formulation is particularly well suited for GPUs, since all steps in \Cref{alg:ssn-simplex} reduce to simple element-wise operations and parallel reductions. Moreover, with the initialization $\tau_0$ defined above, one can show that $J_k > 0$ for all $k \ge 0$, and that the semismooth Newton iteration enjoys finite termination due to the piecewise-linear structure of $\phi(\tau)$ even when we set the tolerance $\epsilon = 0$.

We evaluate three classes of games, each motivated by distinct structural or practical considerations:
\begin{enumerate}[(i)]
\item Random Sparse Game (Ran).
The payoff matrix is generated by
\[
A_{ij}=
\begin{cases}
r_{ij}, & \text{with probability } p_{\rm nz}=0.4,\\
0, & \text{otherwise},
\end{cases}
\qquad r_{ij}\sim \mathcal{N}(0,100).
\]

\item Cyclic Dominance Game (Cyc).
Let $k=\lfloor n/64\rfloor$. The payoff matrix is
\[
A_{ij}=
\begin{cases}
r_{ij}, & (i-j)\bmod n\in\{1,\dots,k\},\\
-r_{ij}, & (j-i)\bmod n\in\{1,\dots,k\},\\
0, & \text{otherwise},
\end{cases}
\qquad r_{ij}\sim \mathcal{U}[0,1].
\]

\item Logistic Distance Game (Log).
The payoff matrix is defined by
\[
A_{ij}=\frac{1.5}{1+\exp(-|i-j|)}+0.5\,\epsilon_{ij}-1,
\qquad \epsilon_{ij}\sim \mathcal{U}[0,1].
\]
\end{enumerate}

Table~\ref{tab:bimatrix-main-split} reports the performance of
MAEG-u, MAEG-y, CFEG, and MFBS on three families of bimatrix games
under two choices of the set $X$. We first compare the effect of $X$.
For CFEG, the forward-evaluation counts are almost identical under the
two choices of $X$ on every instance where both runs succeed. In the
large-scale regime $n\ge 2^{13}$, the runtime increase caused by the
additional projection is at most about $9.6\%$. This indicates that
projection onto the simplex incurs negligible overhead in practice,
which provides direct evidence of the efficiency of the semismooth
Newton projection routine in \Cref{alg:ssn-simplex}. For the
line-search-based MAEG variants, the effect of $X$ is more
substantial: especially on the Log families, using
$X=\Delta_n\times\Delta_n$ leads to significantly fewer forward evaluations
than using $X=\mathbb{R}^n\times\mathbb{R}^n$.

Against the baselines, both MAEG variants are clearly superior. Even
if each baseline is allowed to use its better of the two choices of
$X$, one of the two MAEG variants is still the fastest method and uses
the fewest forward evaluations on every tested instance. The contrast
is sharpest on the Log family: when $n=2^{13}$, MAEG-u terminates in
$6.6$s with only $9.3\times 10^3$ forward evaluations, whereas the
best baseline requires $21.3$s and $2.9\times 10^4$ forward
evaluations. Numerically, MAEG-u and MAEG-y are close, but MAEG-u is
usually slightly better on medium- and large-scale instances.
Restricting attention to the cases $n\ge 2^{13}$ with
$X=\Delta_n\times\Delta_n$, MAEG-u reduces the running time by about
$4.4\%$--$26.1\%$ and the number of forward evaluations by about
$6.1\%$--$17.1\%$ relative to MAEG-y. Thus, the moving-anchor restart
provides stronger convergence guarantees, as shown in \Cref{thm:maeg-r-convergence} and \Cref{rmk:convergence}, without sacrificing empirical
efficiency.

\begin{table}[!htbp]
\centering
\scriptsize
\setlength{\tabcolsep}{5pt}
\renewcommand{\arraystretch}{1.08}
\caption{Comparison of wall-clock time and forward evaluations for the bimatrix game experiments under two choices of the set $X$. The symbol ``-'' indicates that the algorithm failed to satisfy the stopping criterion within the 1-hour time limit.}
\label{tab:bimatrix-main-split}
\begin{tabular}{cccccccccc}
\toprule
\multicolumn{2}{c}{\textbf{Problem}} & \multicolumn{4}{c}{\textbf{Wall-clock time (s)}} & \multicolumn{4}{c}{\textbf{Forward evaluations}} \\
\cmidrule(lr){1-2}\cmidrule(lr){3-6}\cmidrule(lr){7-10}
Game & $n$ & MAEG-u & MAEG-y & CFEG & MFBS & MAEG-u & MAEG-y & CFEG & MFBS \\
\midrule
\multicolumn{10}{c}{$X=\Delta_n\times\Delta_n$} \\
\midrule
\multirow{6}{*}{Ran}
& $2^{10}$ & 3.1 & \textbf{2.8} & 3.3 & 230.9 & 4.3E4 & \textbf{3.7E4} & 5.5E4 & 3.0E6 \\
& $2^{11}$ & \textbf{4.3} & 4.9 & 6.2 & 94.9 & \textbf{3.6E4} & 4.5E4 & 6.4E4 & 8.2E5 \\
& $2^{12}$ & 8.7 & \textbf{6.7} & 11.5 & 600.2 & 4.8E4 & \textbf{3.7E4} & 7.0E4 & 3.2E6 \\
& $2^{13}$ & \textbf{37.0} & 38.7 & 50.9 & 445.8 & \textbf{7.5E4} & 8.0E4 & 1.1E5 & 9.0E5 \\
& $2^{14}$ & \textbf{107.3} & 145.2 & 191.5 & 471.1 & \textbf{6.1E4} & 7.4E4 & 1.1E5 & 2.6E5 \\
& $2^{15}$ & \textbf{459.4} & 551.7 & 846.0 & 3295.1 & \textbf{6.4E4} & 7.6E4 & 1.2E5 & 4.5E5 \\
\midrule
\multirow{6}{*}{Cyc}
& $2^{10}$ & \textbf{1.3} & 1.7 & 1.7 & 19.2 & \textbf{1.8E4} & 2.2E4 & 3.1E4 & 2.7E5 \\
& $2^{11}$ & \textbf{1.7} & 1.8 & 2.8 & 10.8 & \textbf{2.3E4} & 2.4E4 & 4.6E4 & 1.4E5 \\
& $2^{12}$ & \textbf{2.9} & 2.9 & 5.0 & 20.0 & \textbf{3.0E4} & 3.5E4 & 6.9E4 & 2.2E5 \\
& $2^{13}$ & \textbf{6.9} & 7.6 & 13.7 & 29.8 & \textbf{4.6E4} & 5.2E4 & 1.1E5 & 2.0E5 \\
& $2^{14}$ & \textbf{15.5} & 17.8 & 36.3 & 83.8 & \textbf{5.8E4} & 6.9E4 & 1.5E5 & 3.2E5 \\
& $2^{15}$ & \textbf{60.7} & 64.4 & 138.4 & 284.3 & \textbf{8.7E4} & 9.3E4 & 2.1E5 & 4.0E5 \\
\midrule
\multirow{6}{*}{Log}
& $2^{10}$ & \textbf{0.4} & 0.8 & 42.5 & 1.1 & \textbf{6.8E3} & 8.0E3 & 1.1E6 & 1.8E4 \\
& $2^{11}$ & \textbf{1.0} & 1.1 & 163.5 & 2.5 & \textbf{1.0E4} & 1.1E4 & 2.0E6 & 2.3E4 \\
& $2^{12}$ & \textbf{1.7} & 2.1 & 481.3 & 6.1 & \textbf{7.8E3} & 9.7E3 & 2.4E6 & 2.5E4 \\
& $2^{13}$ & \textbf{6.6} & 7.2 & 2365.6 & 21.3 & \textbf{9.3E3} & 1.0E4 & 3.6E6 & 2.9E4 \\
& $2^{14}$ & \textbf{23.0} & 25.2 & - & 55.2 & \textbf{9.1E3} & 9.9E3 & - & 2.2E4 \\
& $2^{15}$ & \textbf{79.8} & 94.5 & - & 223.2 & \textbf{8.0E3} & 9.5E3 & - & 2.2E4 \\
\midrule
\multicolumn{10}{c}{$X=\mathbb{R}^n\times\mathbb{R}^n$} \\
\midrule
\multirow{6}{*}{Ran}
& $2^{10}$ & 4.3 & \textbf{2.7} & 3.0 & 204.1 & 4.6E4 & \textbf{3.8E4} & 5.5E4 & 3.0E6 \\
& $2^{11}$ & \textbf{4.3} & 4.7 & 5.8 & 105.5 & \textbf{4.2E4} & 4.5E4 & 6.4E4 & 9.9E5 \\
& $2^{12}$ & 9.9 & \textbf{6.7} & 11.1 & 528.4 & 5.7E4 & \textbf{3.8E4} & 7.0E4 & 3.0E6 \\
& $2^{13}$ & \textbf{31.5} & 39.4 & 49.2 & 478.3 & \textbf{6.5E4} & 8.2E4 & 1.1E5 & 9.9E5 \\
& $2^{14}$ & \textbf{119.6} & 130.5 & 193.6 & 480.9 & \textbf{6.8E4} & 7.4E4 & 1.1E5 & 2.7E5 \\
& $2^{15}$ & \textbf{460.2} & 556.9 & 849.6 & 3138.0 & \textbf{6.4E4} & 7.7E4 & 1.2E5 & 4.3E5 \\
\midrule
\multirow{6}{*}{Cyc}
& $2^{10}$ & \textbf{1.3} & 1.6 & 1.5 & 21.0 & \textbf{2.0E4} & 2.4E4 & 3.1E4 & 3.3E5 \\
& $2^{11}$ & \textbf{1.8} & 1.9 & 2.4 & 13.0 & \textbf{2.7E4} & 2.9E4 & 4.6E4 & 1.9E5 \\
& $2^{12}$ & \textbf{2.8} & 3.3 & 4.5 & 27.2 & \textbf{3.8E4} & 4.4E4 & 6.9E4 & 3.4E5 \\
& $2^{13}$ & 9.6 & \textbf{9.1} & 12.5 & 37.8 & 7.3E4 & \textbf{7.2E4} & 1.1E5 & 2.9E5 \\
& $2^{14}$ & \textbf{22.5} & 24.9 & 34.2 & 138.9 & \textbf{8.8E4} & 9.5E4 & 1.5E5 & 5.4E5 \\
& $2^{15}$ & \textbf{81.3} & 93.6 & 133.3 & 464.8 & \textbf{1.2E5} & 1.4E5 & 2.1E5 & 6.8E5 \\
\midrule
\multirow{6}{*}{Log}
& $2^{10}$ & 51.1 & 57.4 & \textbf{38.5} & 802.9 & \textbf{9.9E5} & 1.1E6 & 1.1E6 & 1.6E7 \\
& $2^{11}$ & 163.2 & 201.4 & \textbf{154.5} & - & \textbf{1.7E6} & 2.1E6 & 2.0E6 & - \\
& $2^{12}$ & \textbf{550.8} & 638.0 & - & - & \textbf{2.4E6} & 2.5E6 & - & - \\
& $2^{13}$ & - & - & - & - & - & - & - & - \\
& $2^{14}$ & - & - & - & - & - & - & - & - \\
& $2^{15}$ & - & - & - & - & - & - & - & - \\
\bottomrule
\end{tabular}
\end{table}

\subsection{p-Laplacian Obstacle Problem}

To further evaluate the performance of the proposed algorithm, we consider the discretized $p$-Laplacian obstacle problem with $p>2$, which models the equilibrium of a nonlinear elastic membrane subject to a unilateral constraint.

Let $\Omega = (0,1)^2$ be the domain. The problem is modeled as the minimization of the potential energy subject to physical constraints, which is stated as
\begin{equation}\label{eq:continuous-energy}
    \min_{u \in \mathcal{K}} \mathcal{J}(u) = \int_{\Omega} \left( \frac{1}{p} |\nabla u|^p - fu \right) \, dx,
\end{equation}
where $\mathcal{K} = \{ v \in W^{1,p}_0(\Omega) \mid v \ge \psi \text{ a.e.} \}$ is the convex set of feasible configurations. We employ a finite difference discretization on a regular grid of size $N \times N$, with spacing $h = 1/(N-1)$. Let $\mathbf{u} \in \mathbb{R}^{N\times N}$ denote the vector of nodal values. We approximate the continuous energy with the discrete sum
\begin{equation} \label{eq:discrete-energy}
    J_h(\mathbf{u}) = h^2 \sum_{i,j} \left[ \frac{1}{p} \left( \left| \frac{u_{i+1,j}-u_{i,j}}{h} \right|^p + \left| \frac{u_{i,j+1}-u_{i,j}}{h} \right|^p \right) - f_{i,j} u_{i,j} \right],
\end{equation}
where the sums involving forward differences are taken only over valid horizontal and vertical grid edges. The constraints are enforced by the discrete feasible set
\[
\mathbf{K}=\left\{\mathbf{v}\in\mathbb{R}^{N\times N}\mid\mathbf{v}\ge \boldsymbol{\psi}\text{ on interior nodes},\mathbf{v}|_{\partial\Omega}=0 \right\}.
\]

Since $J_h$ is convex and differentiable (for $p \ge 2$), and $\mathbf{K}$ is a non-empty convex set, the necessary and sufficient condition for optimality is given by the monotone inclusion problem~\eqref{eq:mi} with $F(\mathbf{u}) = \nabla J_h(\mathbf{u})$ and $B(\mathbf{u}) = N_{\mathbf{K}}(\mathbf{u}) $, where $F$ is monotone and locally Lipschitz continuous when $p>2$  and $B$ is maximal monotone.

In our numerical experiments, we examine two specific test cases characterized by different obstacle configurations $\psi(x,y)$ and constant external forces $f$:
\begin{enumerate}[(i)]
\item Off-Center Gaussian: The obstacle is a single Gaussian peak centered at $(0.3, 0.6)$ with a constant forcing term $f = -20$. The obstacle is defined as:
\begin{equation*}
    \psi(x,y) = 0.5 \exp\left( -\frac{(x-0.3)^2 + (y-0.6)^2}{0.05} \right).
\end{equation*}
    
\item Double Hump: This case features a ``camel back'' obstacle composed of two Gaussian hills, testing the membrane's deformation into the valley between peaks under a stronger force $f = -30$. The obstacle is given by:
\begin{equation*}
    \psi(x,y) = 0.5 \left[ \exp\left( -\frac{(x-0.3)^2 + (y-0.5)^2}{0.05} \right) + \exp\left( -\frac{(x-0.7)^2 + (y-0.5)^2}{0.05} \right) \right].
\end{equation*}
\end{enumerate}

Table~\ref{tab:plaplacian} reports the performance of MAEG-u, MAEG-y, and MFBS on these two families of $p$-Laplacian obstacle problems. The main message is that the moving-anchor restart substantially improves the practical behavior of the MAEG framework on this problem class. While MAEG-y is consistently slower than MFBS on every instance where both methods terminate, MAEG-u is faster than MFBS on all but the two easiest cases. Restricting attention to the successful instances with $p\ge 4.0$, MAEG-u reduces the wall-clock time by about $26.1\%$--$56.7\%$ and the number of forward evaluations by about $29.4\%$--$56.9\%$ relative to MFBS. The advantage over MAEG-y is even larger, which highlights the practical benefit of the moving-anchor restart on these difficult nonlinear obstacle problems.

\begin{table}[!htbp]
\centering
\caption{Comparison of wall-clock time and forward evaluations for the $p$-Laplacian experiments. The symbol ``-'' indicates that the algorithm failed to satisfy the stopping criterion within the 1-hour time limit.}
\label{tab:plaplacian}
\scriptsize
\setlength{\tabcolsep}{5pt}
\renewcommand{\arraystretch}{1.08}
\begin{tabular}{ccccccccc}
\toprule
\multicolumn{3}{c}{\textbf{Problem}} & \multicolumn{3}{c}{\textbf{Wall-clock time (s)}} & \multicolumn{3}{c}{\textbf{Forward evaluations}} \\
\cmidrule(lr){1-3}\cmidrule(lr){4-6} \cmidrule(lr){7-9}
Name & $N$ & $p$ & \multicolumn{1}{c}{MAEG-u} & \multicolumn{1}{c}{MAEG-y} & \multicolumn{1}{c}{MFBS} & \multicolumn{1}{c}{MAEG-u} & \multicolumn{1}{c}{MAEG-y} & \multicolumn{1}{c}{MFBS} \\
\midrule

\multirow{6}{*}{\shortstack{Off-Center\\Gaussian}}
& $2^8$ & 3.5 & 9.3 & 19.5 & \textbf{8.8} & \textbf{2.2E5} & 5.2E5 & \textbf{2.2E5} \\
& $2^8$ & 4.0 & \textbf{21.3} & 112.5 & 49.2 & \textbf{5.6E5} & 3.0E6 & 1.3E6 \\
& $2^8$ & 4.5 & \textbf{97.7} & 504.8 & 225.8 & \textbf{2.6E6} & 1.3E7 & 6.0E6 \\
& $2^9$ & 3.5 & \textbf{63.8} & 271.8 & 117.1 & \textbf{9.3E5} & 4.0E6 & 1.8E6 \\
& $2^9$ & 4.0 & \textbf{471.8} & 1877.0 & 835.0 & \textbf{6.8E6} & 2.7E7 & 1.2E7 \\
& $2^9$ & 4.5 & \textbf{2721.2} & - & - & \textbf{3.9E7} & - & - \\
\midrule

\multirow{6}{*}{\shortstack{Double\\Hump}}
& $2^8$ & 3.5 & 5.1 & 10.4 & \textbf{4.3} & 1.4E5 & 2.8E5 & \textbf{1.2E5} \\
& $2^8$ & 4.0 & \textbf{22.5} & 83.7 & 35.8 & \textbf{6.0E5} & 2.2E6 & 9.6E5 \\
& $2^8$ & 4.5 & \textbf{86.4} & 417.2 & 185.7 & \textbf{2.3E6} & 1.1E7 & 4.9E6 \\
& $2^9$ & 3.5 & \textbf{43.4} & 143.8 & 56.8 & \textbf{6.2E5} & 2.1E6 & 8.6E5 \\
& $2^9$ & 4.0 & \textbf{383.9} & 1423.9 & 623.6 & \textbf{5.5E6} & 2.1E7 & 9.3E6 \\
& $2^9$ & 4.5 & \textbf{2508.8} & - & 3394.1 & \textbf{3.6E7} & - & 5.1E7 \\
\bottomrule
\end{tabular}
\end{table}

\subsection{Quadratic Minimization with \texorpdfstring{$\ell_2^p$}{L2p}-Regularization}
Consider the convex optimization problem
\begin{equation} \label{prob:regularized-qp}
\begin{aligned}
&\min_{x} \quad \frac{1}{2}x^\top Hx - h^\top x + \frac{1}{p}\|x\|_2^p, \\
&\text{ s.t.} \quad Ax=b,
\end{aligned}
\end{equation}
where the problem data $(A,H,b,h)$ are taken from the hard instance introduced in~\cite{ouyang2021lower}. In particular, \eqref{prob:regularized-qp} is obtained by adding the $\ell_2^p$-regularization term to that benchmark instance.
The above problem can be written as the monotone equation
\begin{equation*}
    0=F(x,\lambda):=\left(Hx-h+A^\top\lambda + \|x\|_2^{p-2}x, b - Ax\right).
\end{equation*}
When $1<p<2$, the operator $F(\cdot)$ is only continuous, not even locally Lipschitz continuous.

Table~\ref{tab:xu} reports the performance of MAEG-u and MFBS for
dimensions $n$ ranging from $2^{10}$ to $2^{15}$ and regularization
parameters $p \in \{1.05, 1.1, 1.15\}$. In this regime, the
Lipschitz-based accelerated algorithms considered above are no longer
applicable. By contrast, MAEG-u remains applicable, with convergence
guaranteed by Theorem~\ref{thm:maeg-r-convergence}.

From the computational point of view, MAEG-u outperforms MFBS on all
tested instances. On the instances where both methods terminate,
MAEG-u reduces the wall-clock time by about $55.5\%$--$96.0\%$ and the
number of forward evaluations by about $56.7\%$--$96.1\%$ relative to
MFBS, with the gap widening as $p$ approaches $1$. In the stiffest
setting $p=1.05$, MFBS fails to satisfy the stopping criterion within
the one-hour time limit for all dimensions $n \ge 2^{13}$, whereas
MAEG-u still solves these instances within $421.1$s. This shows that
the moving-anchor restart keeps the MAEG framework effective and
computationally advantageous even in this genuinely non-Lipschitz
regime.

\begin{table}[!htbp]
\centering
\scriptsize
\setlength{\tabcolsep}{5pt}
\caption{Comparison of wall-clock time and forward evaluations for the $\ell_2^p$-regularized quadratic minimization problem. The symbol ``-'' indicates that the algorithm failed to satisfy the stopping criterion within the 1-hour time limit.}
\label{tab:xu}
\renewcommand{\arraystretch}{1.08}
\begin{tabular}{cccccc}
\toprule
\multicolumn{2}{c}{\textbf{Problem}} & \multicolumn{2}{c}{\textbf{Wall-clock time (s)}} & \multicolumn{2}{c}{\textbf{Forward evaluations}} \\
\cmidrule(r){1-2} \cmidrule(lr){3-4} \cmidrule(l){5-6}
$p$ & $n$ & MAEG-u & MFBS & MAEG-u & MFBS \\
\midrule
\multirow{6}{*}{1.05} & $2^{10}$ & \textbf{15.6} & 180.4 & \textbf{1.4E5} & 1.8E6 \\
& $2^{11}$ & \textbf{29.4} & 503.9 & \textbf{2.8E5} & 4.8E6 \\
& $2^{12}$ & \textbf{54.9} & 1375.4 & \textbf{5.1E5} & 1.3E7 \\
& $2^{13}$ & \textbf{110.4} & - & \textbf{1.0E6} & - \\
& $2^{14}$ & \textbf{223.1} & - & \textbf{2.0E6} & - \\
& $2^{15}$ & \textbf{421.1} & - & \textbf{3.6E6} & - \\
\cmidrule{1-6}
\multirow{6}{*}{1.10} & $2^{10}$ & \textbf{14.3} & 67.9 & \textbf{1.4E5} & 6.7E5 \\
& $2^{11}$ & \textbf{26.9} & 177.3 & \textbf{2.6E5} & 1.7E6 \\
& $2^{12}$ & \textbf{54.4} & 447.8 & \textbf{5.1E5} & 4.2E6 \\
& $2^{13}$ & \textbf{110.0} & 1134.1 & \textbf{1.0E6} & 1.0E7 \\
& $2^{14}$ & \textbf{204.2} & 2825.8 & \textbf{1.8E6} & 2.5E7 \\
& $2^{15}$ & \textbf{418.3} & - & \textbf{3.6E6} & - \\
\cmidrule{1-6}
\multirow{6}{*}{1.15} & $2^{10}$ & \textbf{13.3} & 29.9 & \textbf{1.3E5} & 3.0E5 \\
& $2^{11}$ & \textbf{27.0} & 72.7 & \textbf{2.6E5} & 6.9E5 \\
& $2^{12}$ & \textbf{49.8} & 174.7 & \textbf{4.6E5} & 1.6E6 \\
& $2^{13}$ & \textbf{100.0} & 416.5 & \textbf{9.2E5} & 3.8E6 \\
& $2^{14}$ & \textbf{201.8} & 988.0 & \textbf{1.8E6} & 8.8E6 \\
& $2^{15}$ & \textbf{415.4} & 2379.4 & \textbf{3.6E6} & 2.0E7 \\
\bottomrule
\end{tabular}
\end{table}

\section{Conclusion}
This paper presented the Moving-Anchored Extra-Gradient (MAEG) method for solving monotone inclusion problems. Under Lipschitz continuity of the forward operator, MAEG attains an $\mathcal{O}(1/k)$ non-asymptotic complexity, and for $0<\rho<\frac{1}{2}$, it further achieves an $o(1/k)$ asymptotic rate. We also developed a restart strategy based on the non-increasing distance between the moving anchor point and the solution set. This strategy ensures convergence for continuous monotone operators even without local Lipschitz continuity, while preserving the corresponding Lipschitz-case complexity guarantees after finitely many restarts. Numerical experiments also confirm the practical effectiveness of the proposed method.

\bibliographystyle{plain}
\bibliography{ref}

\begin{appendices}
\section{A Counterexample to the Claim in \texorpdfstring{\cite{khobotov1987modification}}{Khobotov}}

Consider the following variational inequality problem:
\begin{equation*}
    \text{Find } x^{\star} \in Q \text{ such that } \left\langle T(x^\star), x-x^\star\right\rangle\geq 0, \quad \forall x \in Q,
\end{equation*}
where $Q\subset \mathbb{R}^n$ is a nonempty, convex and closed set, and $T:\mathbb{R}^n\to \mathbb{R}^n$ is continuous and monotone on $Q$.

Following the formulation in~\cite{khobotov1987modification}, let $x_\star$ be an arbitrary solution to the problem. We define the following sets depending on $x_\star$:
\begin{align*}
    \Bar{R}_0(x_\star) &= \{x \in \mathbb{R}^n \mid \|x-x_\star\|\leq\|x_0-x_\star\|\},\quad R_0(x_\star)= Q \cap \Bar{R}_0(x_\star), \\
    \hat{R}_0(x_\star) &= \left\{\bar{x} \in \mathbb{R}^n \mid \bar{x}=\Pi_Q(x-\tilde{\alpha} \tilde{b}_0), \, x\in R_0(x_\star), \, \tilde{\alpha}\in [0,\bar{\alpha}],\, \tilde{b}_0\in M_0 \right\},
\end{align*}
where $ \bar{\alpha} > 0 $ is a given parameter and $M_0 = \big\{ z\in \mathbb{R}^n \mid \|z\|\leq \sup\nolimits_{x\in R_0(x_\star)} \|T(x)\| \big\}$.
    
In~\cite{khobotov1987modification}, it is asserted that there exists a constant $L_0<\infty$ such that the Lipschitz condition $\|T(x)-T(y)\|\leq L_0\|x-y\|$ holds for all $x, y \in \hat{R}_0(x_\star)$. We now construct a counterexample to demonstrate that this statement may fail if local Lipschitz continuity of the operator $T$ is not explicitly assumed.

\begin{example}
    Let $n=1$ and $Q=\mathbb{R}$. Define the operator $T$ as:
    \begin{equation*}
        T(x)=\begin{cases}
            1+\sqrt{x}, & x\ge 0, \\
            1-\sqrt{-x}, & x<0.
        \end{cases}
    \end{equation*}
    The operator $T$ is continuous and monotone. The variational inequality has a unique solution $x_\star = -1$.
\end{example}

Let the initial point be $x_0=0$. Then we have $R_0(x_\star)=\Bar{R}_0(x_\star)=[-2,0]$ and $[-2,0]\subseteq\hat{R}_0(x_\star)$. Consider the points $y_0=0$ and $y_k=-\frac{1}{k}$ for $k \in \mathbb{N}$. Both $y_0$ and $y_k$ lie within $\hat{R}_0(x_\star)$. However,
\begin{equation*}
    \frac{\|T(y_0)-T(y_k)\|}{\|y_0-y_k\|} = \frac{|1 - (1-\sqrt{1/k})|}{|1/k|} = \frac{\sqrt{1/k}}{1/k} = \sqrt{k} \to +\infty \text{ as } k \to +\infty.
\end{equation*}
Consequently, there is no finite constant $L_0$ satisfying the Lipschitz condition on $\hat{R}_0(x_\star)$, proving the assertion from~\cite{khobotov1987modification} false in the general continuous case.
\end{appendices}
\end{document}